\documentclass[11pt]{article}
\usepackage[show]{ed}
\usepackage{amsthm}
\usepackage{amsmath}
\usepackage{amsfonts}
\usepackage{latexsym}
\usepackage{amssymb}
\usepackage{graphicx}
\usepackage{algorithm}
\usepackage{algpseudocode}

\newtheorem{theorem}{Theorem}[section]

\newtheorem{corollary}[theorem]{Corollary}
\newtheorem{definition}[theorem]{Definition}
\newtheorem{example}[theorem]{Example}
\newtheorem{lemma}[theorem]{Lemma}

\newtheorem{proposition}[theorem]{Proposition}
\newtheorem{remark}[theorem]{Remark}

\newtheorem{notation}[theorem]{Notation}

\numberwithin{equation}{section}

\usepackage{mathrsfs}
\usepackage{booktabs}

\usepackage{diagbox} 

\usepackage[margin=1.0in]{geometry}    

\newcommand{\secref}[1]{\S\ref{#1}}

\def \Vh0{\stackrel{\circ}{V}_h}

\newcommand{\lc}
{\mathrel{\raise2pt\hbox{${\mathop<\limits_{\raise1pt\hbox
					{\mbox{$\sim$}}}}$}}}

\newcommand{\gc}
{\mathrel{\raise2pt\hbox{${\mathop>\limits_{\raise1pt\hbox{\mbox{$\sim$}}}}$}}}
\newcommand{\ec}
{\mathrel{\raise2pt\hbox{${\mathop=\limits_{\raise1pt\hbox{\mbox{$\sim$}}}}$}}}

\def\bb{\begin{equation}} \def\ee{\end{equation}}
\def\beqn{\begin{eqnarray}}  \def\eqn{\end{eqnarray}}
\def\beqnx{\begin{eqnarray*}} \def\eqnx{\end{eqnarray*}}

\newcommand{\bP}{\mathbb{P}}

\newcommand{\cI}{{\mathcal I}}

\newcommand{\cP}{{\mathcal P}}
\newcommand{\cQ}{{\mathcal Q}}

\newcommand{\cW}{{\mathcal W}}

\newcommand{\ba}{\mathbf{a}}

\newcommand{\ri}{{\rm i}}
\newcommand{\rd}{{\rm d}}

\newcommand{\beq}{\begin{equation}}
	\newcommand{\eeq}{\end{equation}}
\newcommand{\beqs}{\begin{equation*}}
	\newcommand{\eeqs}{\end{equation*}}
\newcommand{\bit}{\begin{itemize}}
	\newcommand{\eit}{\end{itemize}}
\newcommand{\ben}{\begin{enumerate}}
	\newcommand{\een}{\end{enumerate}}
\newcommand{\bal}{\begin{align}}
	\newcommand{\eal}{\end{align}}
\newcommand{\bals}{\begin{align*}}
	\newcommand{\eals}{\end{align*}}
\newcommand{\bse}{\begin{subequations}}
	\newcommand{\ese}{\end{subequations}}
\newcommand{\bpr}{\begin{proposition}}
	\newcommand{\epr}{\end{proposition}}
\newcommand{\bre}{\begin{remark}}
	\newcommand{\ere}{\end{remark}}
\newcommand{\bpf}{\begin{proof}}
	\newcommand{\epf}{\end{proof}}
\newcommand{\ble}{\begin{lemma}}
	\newcommand{\ele}{\end{lemma}}
\newcommand{\bco}{\begin{corollary}}
	\newcommand{\eco}{\end{corollary}}
\newcommand{\bex}{\begin{example}}
	\newcommand{\eex}{\end{example}}
\newcommand{\bth}{\begin{theorem}}
	\newcommand{\enth}{\end{theorem}}

\def\XXint#1#2#3{{\setbox0=\hbox{$#1{#2#3}{\int}$}
		\vcenter{\hbox{$#2#3$}}\kern-.5\wd0}}

\usepackage{float}
\newcommand{\wu}{\widetilde{u}}
\newcommand{\hba}{\widehat{\ba}}
\newcommand{\tba}{\widetilde{\ba}}
\newcommand{\ha}{\widehat{a}}
\newcommand{\ta}{\widetilde{a}}
\newcommand{\hf}{\widehat{f}}

\newcommand{\hg}{\widehat{g}}

\newcommand{\by}{\boldsymbol{y}}
\newcommand{\bR}{\mathbb{R}}

\newcommand{\bell}{\boldsymbol{\ell}}

\newcommand{\bN}{\mathbb{N}}

\definecolor{darkpurple}{RGB}{110,0,180}

\newcommand{\igg}[1]{{\color{red}{#1}}}

\usepackage{hyperref}
\hypersetup{colorlinks=true,linkcolor=blue,citecolor=blue,filecolor=blue,urlcolor=blue}

\providecommand{\blue}{\textcolor{black}}

\renewcommand\d{\mathrm{d}}

\makeatletter
\def\ps@pprintTitle{%
  \let\@oddhead\@empty
  \let\@evenhead\@empty
  \let\@oddfoot\@empty
  \let\@evenfoot\@oddfoot
}
\makeatother

\title{A Filon-Clenshaw-Curtis-Smolyak rule for multi-dimensional\\ oscillatory integrals
  with  application \\ to a UQ problem for the Helmholtz equation\\
}

  \author{Zhizhang Wu$^\dagger$,  Ivan G. Graham$^*$, Dingjiong Ma$^\dagger$,  Zhiwen Zhang$^\dagger$\\[2ex]
  $^\dagger$ Department of Mathematics, The University of Hong Kong, \\Pokfulam Road, Hong Kong SAR, China.
  \\ \tt wuzz@hku.hk, martin35@connect.hku.hk,  zhangzw@hku.hk \\[2ex]
$^*$ Department of Mathematical Sciences, University of Bath, Bath BA2 7AY, UK.  \\{\tt I.G.Graham@bath.ac.uk}
}
\date{\today}

\begin{document}

\maketitle

\begin{abstract}
\noindent
In this paper, we combine the Smolyak technique for multi-dimensional interpolation with the Filon-Clenshaw-Curtis
(FCC) rule for one-dimensional oscillatory integration,  to obtain a new Filon-Clenshaw-Curtis-Smolyak (FCCS) rule for oscillatory integrals with linear phase  over the $d-$dimensional cube $[-1,1]^d$.  By combining stability and convergence estimates for the FCC rule with error estimates for the Smolyak interpolation operator, we obtain an error estimate for the FCCS rule, consisting of the product of a Smolyak-type error estimate multiplied by a term that decreases with $\mathcal{O}(k^{-\widetilde{d}})$, where $k$ is the wavenumber and $\widetilde{d}$ is the number of oscillatory dimensions. If all dimensions are oscillatory, a  higher negative power of $k$ appears in the estimate. As an application, we consider the forward problem of uncertainty quantification (UQ) for a  one-space-dimensional Helmholtz problem with wavenumber $k$ and a random heterogeneous refractive index, depending  in an affine way on $d$ i.i.d. uniform random variables. After applying a classical  hybrid numerical-asymptotic approximation, expectations of  functionals  of the solution of this problem can   be formulated as a sum of oscillatory integrals over  $[-1,1]^d$, which we compute  using the FCCS rule. We give numerical results for the FCCS rule which illustrate its  theoretical properties  and show that    the accuracy of the UQ algorithm improves when  both $k$ and the order of the FCCS rule increase. We also give results for both the quadrature and UQ problems when the underlying  FCCS rule uses a  dimension-adaptive Smolyak interpolation. These show increasing accuracy for the UQ problem as both the adaptive tolerance decreases and $k$ increases, requiring  very modest increase in work as the stochastic dimension increases, for a case when the affine expansion in random variables decays quickly.

\end{abstract}





\section{Introduction} \label{sec:intro}


\noindent
In this paper, we formulate and analyse a novel numerical method for computing the multi-dimensional oscillatory integral
\begin{align}
	{\cI^{k,d,\ba}} f : =  \int_{[-1,1]^d} f( \by) \exp(\ri  k \ba\cdot \by ) \rd \by,
	\label{highDoscIntegralEq}
\end{align}
where $k>0$  is a parameter, which may be large, and $\ba=(a_1,...,a_d)^{\top} \in \mathbb{R}^d$ is a fixed vector. As an application of this,   we solve an uncertainty quantification problem for the Helmholtz equation (modelling frequency-domain linear waves),  via  a hybrid numerical asymptotic method, yielding  increasing accuracy as the frequency  increases.

\paragraph{Background} The computation of oscillatory integrals is a classical problem in
applied mathematics (e.g., \cite{Wong:2001}),  which has   enjoyed considerable recent interest. By combining  numerical and  asymptotic techniques,  research has focussed on providing methods  which work well  for moderate  $k$,  but  remain accurate (or even improve in accuracy) as the parameter $k$ (proportional to frequency) gets \blue{large.}   While there has been strong interest in this topic \blue{ in the recent past} (partly driven by applications in high-frequency scattering e.g., \cite{GrahamSpence:2012, Groth:2018,Gibbs:2020}), most methods proposed in this context are appropriate only
for relatively low-dimensional oscillatory integration problems.

On the other hand there is a considerable literature on multi-dimensional integration for the non-oscillatory version of \eqref{highDoscIntegralEq}, where $k$ is small - here we mention just  \cite{GeGr:98,NoRi:99,BaNoRi:00,DiKuSl:13,NoTaTe:16,ZeSc:20} as exemplars of the huge literature on this topic. However, since the accuracy of these rules depends on the derivatives of the integrand, their direct application  to the whole integrand in \eqref{highDoscIntegralEq} will incur an error which in general will  blow up strongly with increasing $k$.

Quite a large portion of research on  oscillatory integration in the {low-dimensional} case (mostly $d=1$) is concerned with  Filon-type methods. In  \cite{iserles2004numerical,iserles2005numerical,IserlesNorsett:2004,IserlesNorsett:2006,xiang2007efficient} the analysis  concentrates on accelerating the convergence as $k \rightarrow \infty$.

The basic 1D method central to the current paper  is
\cite{GrahamDominguez:2011}, which proves stability and algebraic convergence
\blue{(with respect to the number of function evaluations) for the Filon-Clenshaw-Curtis  rule, explicit in the regularity required of $f$. The convergence is superalgebraic if $f \in C^\infty$ and the error estimate
features a negative power of $k$ as $k \rightarrow \infty$.}  The range of application of this approach was considerably extended  in recent work \cite{MaIsPe:22}.
Extensions of Filon methods to $hp$ approximation and the handling of nonlinear phase functions (again in 1D) are given in
\cite{melenk2010convergence,GrahamDominguez:2013,MaIsPe:22,Ma:21a,Ma:21}.

\blue{A complexity study of quadrature rules for oscillatory integration in one dimension has been carried out
  in \cite{NoUlWo:15}, with a recent review in \cite{No:23}. We comment on the relation of these results to ours in Remark \ref{rem:Novak} later in the paper.} 

To extend  the  approach of \cite{GrahamDominguez:2011}  to the higher dimensional case,  the factor $f(\by)$ in \eqref{highDoscIntegralEq} should be approximated by some linear combination of simple basis functions, with coefficients which can be computed easily from $f$, and then this approximation should be  integrated analytically against the oscillatory factor $\exp(\ri k   \ba\cdot \by)$. There are relatively  few papers on the higher dimensional oscillatory case. Exceptions are \cite{IsNo:06, HuVa:07} \blue{(see also \cite{DeHuIs:18})} which include discussion of  generalization of a  Filon-type method to problem  \eqref{highDoscIntegralEq}, making use of function values and derivatives at vertices of the boundary and proving increasing accuracy as  $k \rightarrow \infty$, but without explicit error estimates showing  how the error depends on the number of function evaluations,   the regularity of $f$ or the dimension $d$.

\paragraph{Overview of Algorithm} Our method for  \eqref{highDoscIntegralEq} essentially extends the  1D  `Filon-Clenshaw-Curtis' approach to the multidimensional case by {applying Smolyak-type interpolation to the non-oscillatory part of the  integrand in  \eqref{highDoscIntegralEq}. Since  we shall  allow $\ba$ to have entries of either sign and possibly small, we introduce the following notation in order  to identify the oscillatory and non-oscillatory dimensions in  \eqref{highDoscIntegralEq}.
   \begin{notation}\label{not:tilde}
    For $a \in \mathbb{R}$ and $k> 0$ we define
    \begin{align*} \ta = \left\{ \begin{array}{ll} a \quad & \text{if } \quad k \vert a \vert \geq 1\\
                                   0 \quad & \text{otherwise} \end{array} \right. \end{align*}
                             and we set  $\ha = a - \ta$.  (We note that $\ta$ and $\ha$ depend on $k$ as well as $a$.)  
                           \end{notation}
                       Applying Notation \ref{not:tilde}
                       to each component of the vector $\ba$ in \eqref{highDoscIntegralEq}
                         \blue{we obtain the decomposition   $k\ba =  k\hba + k\tba   $, where each component of $k \hba$ has modulus bounded above by 1.   and thus  $\hf (\by) := f(\by) \exp( \ri k \hba \cdot \by)$                                                                           
                       is the non-oscillatory part of the integrand in \eqref{highDoscIntegralEq}. }
                       We then  rewrite \eqref{highDoscIntegralEq}  as
                           \begin{align}
	\blue{{\cI^{k,d,\ba}} f  =}  \int_{[-1,1]^d} \hf( \by) \exp(\ri  k \tba \cdot \by ) \rd \by.
	   \label{highDoscIntegralEq1}
                           \end{align}

                                           Our quadrature rule for \eqref{highDoscIntegralEq1} (and hence \eqref{highDoscIntegralEq})  is then defined by replacing    the  factor $\hf$
                                           by its classical  Smolyak polynomial interpolant $\mathcal{Q}^{r,d}\hf$ of maximum level $r$ (formally defined in \eqref{FCCSmolyak-dD}).
                                           This sparse grid interpolant employs
                 separable polynomial interpolation at points on  sparse grids generated by
                                             a nested sequence of  1-d grids. Here we use, at level $\ell$, the points:
                                             \begin{align} \label{eq:grids} \{ 0 \} \ \text{for} \ \ell = 1,  \quad
    \text{and} \ \left\{
         t_{j,\ell} := \cos\left(\frac{j\pi}{n_\ell}\right) \right\}_{j=0}^{n_\ell} \quad \text{for} \  \ell \ge 2, \quad \text{where} \ n_\ell = 2^{\ell-1},  \end{align}
i.e.  the mid-point rule is used at level $\ell = 1$  and $2^{\ell-1} +1$  Clenshaw-Curtis points are used at level $\ell \geq 2$.
Using this,  we  approximate \eqref{highDoscIntegralEq1}  by
\begin{align}
\mathcal{I}^{k,d,\ba,r} f  := \int_{[-1,1]^d}(\mathcal{Q}^{r,d}\hf)(\by)\exp(\ri  k \tba \cdot  \by  ) \rd \by.
\label{FCCSmolyak_intro}
\end{align}
This is a $d$-dimensional version of the 1D Filon-Clenshaw-Curtis (FCC) rule from   \cite{GrahamDominguez:2011}.

Since, in each dimension the interpolant on the Clenshaw-Curtis grid can be written as a
linear combination of Chebyshev polynomials of the first kind  of degree $n$ (here denoted by $T_n$), the integral \eqref{FCCSmolyak_intro} can be computed  exactly, provided  the $k-$dependent `weights'
\begin{align} \label{weights}
W_n(ka_j) := \int_{-1}^{1}T_n(y) \exp(\ri k a_j y) dy, \quad n = 0, \ldots , n_{\ell},            \quad \ell \geq 2 ,  \quad j = 1, \ldots d
\end{align}
are known. A stable algorithm for computing these weights (for any $k$, $a_j$  and $n$), and its analysis,  are
given in \cite{GrahamDominguez:2011}. (The weight for the case $\ell = 1$ is trivial to compute.)  Since the cost of computing the weights for an $M+1$ point rule in 1-d and with a fixed choice of $k$ and $a_j$ has complexity $\mathcal{O}(M\, \log M)$ (see \cite[Remark 5.4]{GrahamDominguez:2011}), and since the weights for each dimension can be computed independently, the cost of computing the weights \eqref{weights} (on a serial computer) grows with  $\mathcal{O}(d \, 2^{r-1} \log (2^{r-1}))$ as dimension $d$ or the maximum level $r$ increases.  Weight computation in each dimension could be done in parallel.


\paragraph{Main  results of the paper}
In order to prove an error estimate, we assume that
\begin{align}
f \in \mathcal{W}^{p,d} := \left\{ f:[-1,1]^d\rightarrow \mathbb{R}: { \frac{\partial^{|\textbf{s}|}f}{\partial \by^{\mathbf{s}}} \in C([-1, 1]^d) }, \text{ for all } | \mathbf{s} |_{\infty} \leq p \right\}
\label{FunctionSpaceSetting}
\end{align}
for some positive integer $p$, where $\Vert \cdot \Vert_\infty $ denotes the uniform norm over $[-1,1]$,
$\textbf{s}=(s_1,...,s_d) \in \mathbb{N}_0^d$ are the multi-indices of order $|\textbf{s}| = s_1 + \cdots + s_d$,
$$
\frac{\partial^{|\textbf{s}|}f}{\partial \by^{\mathbf{s}}} = \frac{\partial^{|\textbf{s}|}f}{\partial y_1^{s_1}\cdots\partial y_d^{s_d}},
$$
and  $| \mathbf{s} |_{\infty} = \max \limits_{1 \le i \le d} s_i$.
\blue{Thus $\mathcal{W}^{p,d}$ is here defined as the space of $d-$variate functions on $[-1,1]^d$, whose mixed partial derivatives of up to order $p$ in each dimension are continuous.} \, We introduce the norm on $\mathcal{W}^{p,d}$:
\begin{align}
  \| f \|_{\mathcal{W}^{p,d}} := \max_{\mathbf{s} \in \mathbb{N}_0^d \, : \, | \mathbf{s} |_{\infty} \le p  }
  \left\| \frac{\partial^{|\textbf{s}|}f}{\partial \by^{\mathbf{s}}} \right\|_{\infty, [-1,1]^d}.  
\end{align}
We note that $\cW^{p,1}$ is just the usual space $C^p[-1,1]$ with norm given by $$\Vert f \Vert_{\cW^{p,1}} =
\max \{ \Vert f^{(j)}\Vert_{\infty, [-1,1]}, \quad j=0,\ldots , p \} . $$
By combining the properties of the Smolyak algorithm and the FCC rule, together with the regularity
assumption \eqref{FunctionSpaceSetting}, we prove in \S \ref{sec:Error_Analysis} the following error estimates.

\begin{theorem}\label{thm:main_grid}
For each  $p\geq 1$ and $d \geq 1$ there is a constant $C_{d,p}>0$ such that,  for all $\ba \in \mathbb{R}^d$, $k > 0$,  and
  $r$ sufficiently large,  we have
  \begin{align}
& \vert \cI^{k,d,\ba} f - \cI^{k,d,\ba,r} f \vert \ \nonumber \\
&\mbox{\hspace{1cm}}  \leq C_{d,p}  \left(\prod_{\stackrel{j \in \{ 1, \ldots, d \}}{k \vert a_j\vert \geq 1 }}
      k \vert a_j \vert   \right)^{-1}
    \left(  \log^{d-1}N(r, d) \right)^{p}  \,
    \left(\frac{1}{N(r, d)}\right)^{p - 1}   \Vert f \Vert_{\cW^{p,d}}, \label{mainalt}
  \end{align}
where $N(r,d)$ is the number of point evaluations of  $f$ used in the quadrature rule  \eqref{FCCSmolyak_intro}.
 \end{theorem}
While this result is explicit in $k$ and $ \ba$, we can obtain the following better estimate, for sufficiently  large  $k$, although this is implicit in its dependence on $\ba$ and requires more regularity on $f$.

\begin{theorem}\label{thm:main_grid1}
  Suppose $\ba \in \mathbb{R}^d$ with $a_j \not = 0$ for each $j$. Then, for each  $p\geq 1$ and  $d \geq 1$,   there is a constant $C_{d,p,\ba}>0$ such that, for \blue{$r$ sufficently large}, we have
  \begin{align}
 \vert \cI^{k,d,\ba} f - \cI^{k,d,\ba,r} f \vert \
   \  \leq\  C_{d,p, \ba}  \, k^{-(d + 1)} \,
    \left(  \log^{d-1}N(r, d) \right)^{p}  \,
    \left(\frac{1}{N(r, d)}\right)^{p - 1}   \Vert f \Vert_{\cW^{{p + 3},d}}, \label{mainalt1}
  \end{align}
\end{theorem}


The third and fourth terms on the right-hand sides of \eqref{mainalt} and \eqref{mainalt1}  are standard in the analysis of Smolyak-type approximation methods, having a power of $\log N$ which grows with dimension. While this is good for moderate (but not very high) dimension $d$, our estimates \eqref{mainalt}, \eqref{mainalt1}  also contain  additional factors containing potentially  high negative powers of $k$, giving a substantial advantage over traditional tensor product rules for these integrals: If  a $d-$dimensional tensor product of the standard one-dimensional  (non-Filon) Clenshaw-Curtis rule were used to approximate  \eqref{highDoscIntegralEq} then (for fixed $p$ and $d$),  the standard error estimate would give  $\mathcal{O}(k^{p}\,  N^{-p/d})$
as $k,N \rightarrow \infty$  (where $N$ is the total number of points used - see Example \ref{prop:bad} for a precise statement). This is   vastly inferior to \eqref{mainalt}, \eqref{mainalt1} when either $d$ or $k$ is large. The estimate   \eqref{mainalt} can be seen as a generalisation of the concept of {\em universality} discussed in \cite[eqn. (7)]{NoRi:99}  to the case of oscillatory integrals.

\blue{When computing oscillatory integrals,  one should always bear in mind that the integral itself is usually  decaying in modulus as $k$ increases. In 1d,   simple arguments (see, e.g., Remark \ref{rem:Novak})
  show that  $\Vert \cI^{k,1,a} \Vert_{\cW^{p,1}} \sim k^{-1}$,  provided $a \not = 0$ and this argument can be extended to show that  $\Vert \cI^{k,d,\ba} \Vert_{\cW^{p,d}} \sim k^{-\widetilde{d}}$,  
  where $\widetilde{d}$ is the number of oscillatory dimensions (i.e., the number of components $a_j$ of $\ba$ such that $k |a_j| \ge 1$). For this reason,  it is natural to also consider the  {\em normalized error}, which, for simplicity, we define here as}
  \begin{align} \label{norm_err} \blue{k^{\widetilde{d}} \,  \vert \cI^{k,d,\ba} f - \cI^{k,d,\ba,r} f \vert} \ ,
  \end{align}
\blue{  (although a slightly different definition of this concept is used in complexity theory -- see Remark \ref{rem:Novak}).} \blue{Then Theorem \ref{thm:main_grid}  shows that the absolute error of the FCCS rule decays with at least  the same order $\mathcal{O}(k^{-\widetilde{d}})$  as $\cI^{k,d,a} f$ itself (as $k \rightarrow \infty$), whereas
  Theorem \ref{thm:main_grid1}   essentially shows that, under certain mild conditions, the
  \blue{normalized error} of the FCCS rule actually decays with at least order $\mathcal{O}(k^{-1})$ as  $k \rightarrow \infty$.
}


\blue{It may be noticed that the estimate \eqref{mainalt} contains the factor $N(r,d)^{1-p}$, for functions with $p$ mixed derivatives,  whereas in the classical  non-oscillatory case (e.g.,   \cite{NoRi:99}), this is replaced by the better factor $N(r,d)^{-p}$.  Since we consider here  the oscillatory case,  it is important to have negative powers of $k$ in the estimate,  and for this we pay with one less
  negative power of $N(r,d)$. Nevertheless we also point out that the following  alternative to \eqref{mainalt} }
    \begin{align}
\blue{ \vert \cI^{k,d,\ba} f - \cI^{k,d,\ba,r} f \vert \   \leq C_{d,p}  
    \left(  \log^{d-1}N(r, d) \right)^{p+1}  \,
    \left(\frac{1}{N(r, d)}\right)^{p}   \Vert f \Vert_{\cW^{p,d}},  \label{mainalt2}}
  \end{align}
  \blue{valid for all $k$, can also be proved, thus  recovering the usual power of $N(r,d)$,
    but losing decay with respect to $k$.  Theorems \ref{thm:main_grid},  and \ref{thm:main_grid1} will be proved in \S \ref{sec:Error_Analysis}, and there we shall also give a sketch of the proof of  \eqref{mainalt2}.}

In proving \eqref{mainalt}, \eqref{mainalt1} we make  no special assumption concerning the decay of the derivatives of $f$ with respect to increasing dimension. Numerical experiments for the non-oscillatory case (e.g.,  \cite{GeGr:03,nobile2016adaptive}) have shown that  if a suitable decay rate is present, then dimension-adaptive algorithms will give better   results as $d$ increases. Theory underpinning the idea of the dimension-adaptive algorithms is given in \cite{chkifa2014high,schillings2013sparse}. In \S \ref{subsec:quad} we give computations using both a standard and a dimension-adaptive method (the latter adapted from the algorithm in \cite{nobile2016adaptive})  for the oscillatory quadrature problem. These  demonstrate the theoretical properties of the former in general and the efficiency of the latter in cases where the importance of the dimensions is decaying.

As an application of our FCCS rule, we consider the UQ   problem for the one-space-dimensional Helmholtz boundary-value problem:
\begin{align}
Lu(x)&:=u''(x)+k^2n^2(x)u(x)=F(x), \quad 0<x<1    \label{Helmholtz_Eq_2}\\
B_Lu(x)&:=u(0)=u_L,                         \label{Helmholtz_BC_1}\\
B_Ru(x)&:=u'(1)-ikn_{\infty}u(1)=0,         \label{Helmholtz_BC_2}
\end{align}
where $k$ is the wave number, $u_L$ and $n_{\infty}>0$ are constants, $ F $ is a smooth function and $n$ is the (generally variable) refractive index, here assumed  sufficiently smooth \blue{and uniformly bounded  above and below by positive numbers.}
The boundary conditions model a sound-soft scattering boundary at $x=0$ and a simple absorbing boundary condition at $x = 1$. In our UQ problem,
 $n$  is assumed to be a random field of the form
\begin{align} \label{eq: temp_par}
n(x,\by)=n_0(x) + \sum_{j=1}^{d}n_j(x)y_j,
\end{align}
where     $\by\in [-1,1]^d$ are uniform i.i.d. random variables and the quantity of interest (QoI) is a linear   functional (with respect to the $x$ variable) of the solution $u$. Although 1-d in space, this problem still has some considerable difficulties for large $k$, because the solution $u(x, \by)$ suffers oscillations with respect not only to the spatial variable $x$ but also to the random variable $\by$ as $k\rightarrow \infty$, and the latter phenomenon poses  considerable difficulty for UQ algorithms at high wavenumber. The structure of these oscillations  is explained in more detail in \S \ref{subsec:random}.

However (because we are in 1d in space), for each fixed choice of $\by$, the resulting deterministic Helmholtz problem can be solved (with accuracy up to any negative order in  $k$) using an  asymptotic method (with some numerical approximation),  originally proposed by Aziz, Kellogg and Stephens \cite{aziz1988two},  in which the work involved is independent of $k$. 
     This provides us with an asymptotic ansatz for the solution of the random Helmholtz problem, and (it turns out that) the expected value of the QoI 
     then can be expressed as a  sum of oscillatory and non-oscillatory  integrals of the form \eqref{highDoscIntegralEq} with respect to the random variables, which  we can compute with $k-$independent accuracy using  our FCCS rule. 

In \S \ref{subsec: numerical examples of UQ problem}   we give numerical results for the Helmholtz UQ problem, using the the numerical-asymptotic approximation for the Helmholtz problem and comparing the standard and dimension-adaptive methods for the multi-dimensional integrals. In the dimension-adaptive case,  for  an example where the expansion \eqref{eq: temp_par} decays exponentially,  we observe increasing accuracy of the results as both the adaptive tolerance decreases and the wave number $k$ increases, with very modest growth in complexity with increasing  dimension.

  This is significant, since it is known that for Helmholtz problems in any space dimension,
  derivatives with respect to the random parameters $\by$ of the solution
   blow up as  $k$ increases, thus enforcing strong  constraints on UQ methods in general at high wavenumber. For example, in    \cite{GaKuSl:21},  conditions ensuring  convergence of first order randomized Quasi-Monte Carlo methods for a Helmholtz problem in any space dimension were studied. There, to ensure a  dimension-independent optimal
  error estimate, one requires $\sum_{j=1}^\infty \Vert n_j\Vert_{W^{1,\infty}}^{2/3} = \mathcal{O}(k^{-2/3})$, (i.e., the amplitude of the allowed  randomness must decrease as $k$ increases).  Strong constraints on the allowed amplitude of the randomness also appear in the so-called multi-modes method described in \cite{FengLinLorton:2015}. We impose no such constraint in our computations. Quasi-Monte Carlo methods for random Helmholtz problems in 2d were studied in  \cite[Chapter 4]{Pe:20}, where it was observed (for moderate wavenumbers)  that the number of quadrature points needed to ensure a bounded error as $k$ increased apparently grew exponentially in $k$.

  The blow-up (as $k$ increases) of the derivatives with respect to $\by$  of the Helmholtz solution is directly related to the width of the region in $\mathbb{C}^d$ in which (the complex extension of)  $u$ is holomorphic.
This region of holomorphy is analysed in detail in the recent paper \cite{SpWu:22} for trapping and non-trapping geometries in any dimension and with general random perturbations of a deterministic base problem. In particular it is shown there  that the estimates in \cite{GaKuSl:21} (for a non-trapping case and the expansion \eqref{eq: temp_par}) are sharp.





\bigskip

The remainder of the paper is organized as follows. In \secref{sec:FCC_1D}, we recall the Filon-Clenshaw-Curtis (FCC)  rule for the 1d oscillatory integral and give some basic theory for it .
In \secref{sec:FCCS_high_dimension}, we combine the 1d quadrature with the Smolyak algorithm to obtain our new FCCS rule for the multi-dimensional oscillatory integrals.
In \secref{sec:Error_Analysis}, we give the error analysis for the  FCCS rule, proving Theorems \ref{thm:main_grid} and \ref{thm:main_grid1}. The application to the  UQ problem for the  Helmholtz equation
is given in \secref{sec:Apply_Helmholtz}. In \secref{sec:Numerical}, we present  numerical examples to demonstrate the performance of the FCCS rule and its application to the UQ problem, \blue{while \S \ref{sec:conclusion} provides  a short concluding section.}


\section{The FCC quadrature rule for 1D problems}
\label{sec:FCC_1D}
\noindent
In this section, we briefly review from \cite{GrahamDominguez:2011} the Filon-Clenshaw-Curtis (FCC) rule for approximating the one-dimensional  integral
\begin{align}
I^{\omega} g : =  \int_{-1}^{1} g( y) \exp(\ri \omega y) dy, \label{1DoscIntegralEq}
\end{align}
for  $\omega \in \mathbb{R}$, {where $g \in {C^{p}[-1,1]}$ for some integer $p \geq 1$.


In the oscillatory case ($\vert \omega \vert \geq 1$), the  FCC quadrature rule is obtained by replacing the factor $g$ in  \eqref{1DoscIntegralEq} by its degree $N$ polynomial interpolant at the Clenshaw-Curtis points $ \cos(j\pi / N)$, \
$j=0,...,N$ for $N \ge 1$ (extrema of the Chebyshev polynomial of the first kind   $ T_N(y):=\cos\big(N\arccos(y)\big)$).
The  interpolant is expressed in terms of the basis $\{T_n: n = 0, \ldots, N\}$,  and the products of these basis functions  with the oscillatory function $\exp(i\omega y)$ are integrated exactly to obtain the quadrature weights.

Starting with the  nested set of quadrature points \eqref{eq:grids}, the FCC approximation to \eqref{1DoscIntegralEq} is then
\begin{align}
{I^{\omega,\ell}_{\mathrm{FCC}} g} : =  \int_{-1}^{1} (Q^\ell g)(y) \exp(\ri \omega y) dy, \quad
\label{1DoscIntegralEq-FCC}
\end{align}
where $(Q^1 g) = g(0)$ and, for $\ell \ge 2$,  $Q^\ell g$ is the polynomial of degree $n_\ell $ satisfying
\begin{align}
(Q^\ell g)(t_{j,\ell}) = g(t_{j,\ell}), \quad j = 0, \ldots, n_\ell. \label{1Dinterpolation-operator}
\end{align}
It is a classical result that, for $\ell \ge 2$, $Q^\ell g$ can be written as
 \begin{align}
(Q^\ell g)(y)=\sum_{n=0}^{n_\ell}\vphantom{\sum}^{''}a_{n,\ell}(g)T_{n}(y), \label{ChebyshevExpansion1D}
\end{align}
where the notation $ \sum\vphantom{\sum}^{''} $ means that the first and the last terms in the sum are to be halved,
and the coefficients $a_{n, \ell}(g)$ are given by the discrete cosine transform:
\begin{align}
a_{n,\ell}(g)\, =\, \frac{2}{n_\ell} \sum_{j=0}^{n_{\ell}}\vphantom{\sum}^{''}\cos\left(\frac{jn\pi}{ n_\ell}\right)g(t_{j,\ell}),\quad n=0, \ldots, n_\ell .
\label{ChebyshevCoefficients}
\end{align}
Substituting \eqref{ChebyshevExpansion1D} into \eqref{1DoscIntegralEq-FCC} for $\ell \geq 2$, we obtain the  quadrature rule
\begin{align}
{I^{\omega,\ell}_{\mathrm{FCC}} \, g} :=
\left\{
\begin{aligned}
& W_0(\omega) g(0), & \quad \ell = 1, \\
& \sum_{n=0}^{n_\ell}  \vphantom{\sum}^{''} W_n(\omega)a_{n,\ell}(g), & \quad \ell \ge 2,
\end{aligned}
\right.
\label{1DoscIntegralEq-FCC-weight}
\end{align}
where the weights 
\begin{align}
W_n(\omega)=\int_{-1}^{1}T_n(y)\exp(\ri \omega y) \, dy, \quad n = 0, \ldots , n_\ell
\label{1DoscIntegralEq-FCC-weight2}
\end{align}
have to be computed.
An algorithm for computing  $W_n(\omega)$ for $\ell \geq 2$ is given and shown to be stable for all $n_\ell$ and $\omega$  in \cite{GrahamDominguez:2011}.

In the case where $\omega = 0$, the weights are known analytically:
\begin{align}
W_n(0)=\int_{-1}^{1}T_n(y)dy =\left\{
\begin{aligned}
& 0, & & n \text{ is odd,}\\
& \frac{2}{1-n^2}, & & n \text{ is even,}
\end{aligned}
\right.
\end{align}
and these provide us with the classical standard  Clenshaw-Curtis (CC) rule  \cite{clenshaw1960method}:
\begin{align} \label{CC}
\int_{-1}^1 g(y) \rd y \ \approx I^\ell_{\mathrm{CC}}\,  g \ :=\ \int_{-1}^1 Q^\ell g \ = \
\left\{
\begin{aligned}
& 2 g(0), & \quad \ell = 1, \\
& \sum_{n=0}^{n_\ell}\vphantom{\sum}^{''}W_n(0) a_{n,\ell}(g), & \quad \ell \ge 2.
\end{aligned}
\right.
\end{align}
When $\vert \omega \vert < 1$ the integral \eqref{1DoscIntegralEq}   is considered to be  non-oscillatory, and can be approximated directly with the CC rule.
Hence  our 1D quadrature method is:
\begin{definition}\label{def:quad}
For $\omega \in \mathbb{R} $ and $\ell \geq 1$, we define the approximation $I^{\omega,\ell} g$ to the integral  \eqref{1DoscIntegralEq} by
\begin{align}
I^{\omega,\ell} g = \left\{ \begin{array}{ll} I_{FCC}^{\omega,\ell} \, g, & \text{when} \ \ \vert \omega\vert  \ge 1, \\ & \\
                             I_{CC}^{\ell} \, (g(\cdot) \exp(\ri \omega \cdot)), & \text{when} \ \  \vert \omega\vert  <  1.  \end{array} \right. \label{eq:def:quad}
\end{align}
\end{definition}
\begin{remark}
Given the values $\{g(t_{j,\ell}): j = 1, \ldots 2^{\ell-1}\}$, the  quadrature rule $I^{\omega,\ell} \, g$  can be computed  with complexity $\mathcal{O}(n_\ell \log n_\ell)$, using FFT (see, e.g.,  \cite[Remarks 2.1, 5.4]{GrahamDominguez:2011}).
  \end{remark}

The following simple proposition uses  Notation \ref{not:tilde}  to give a unified expression for    \eqref{eq:def:quad} and will be useful later.
\begin{proposition}\label{prop:simple} Let $g \in C[-1,1]$ and $\ell \geq 1$.  Then the
  rule defined in \eqref{eq:def:quad} can be written
\begin{align} \label{eq:corres}  I^{ka,\ell} g = \int_{-1}^1 (Q^\ell \widehat{g})(y) \,  \exp(\ri k \ta y) \rd y , \quad \text{where} \quad \widehat{g}(y) = g(y) \exp(\ri k \ha y) .
\end{align}
\end{proposition}

The following error estimate is a minor variation on the result in \cite[Corollary 2.3]{GrahamDominguez:2011}.

\begin{theorem}\label{prop:Errorestimate1dFCC}
   Define $\eta(1) = 0$, $\eta(2) = 3$. Then, for  $p \in \mathbb{N}$, $p >1$, and  $s \in \{ 1,2\} $,   there exists a constant $C_p$ such that, for all $ \ell \geq 2$     and $g \in \cW^{p+ \eta(s),1}$,   
  the quadrature rule \eqref{eq:def:quad} has the error estimate:
\begin{align}
  \Big|I^{\omega} g - {I^{\omega,\ell} g}\Big| \ \leq\  {C_p} \, \min\Big\{1,|\omega|^{-s}\Big\}\,
  \left(\frac{1}{n_\ell} \right)^{p-1} \, \Vert g \Vert_{\cW^{p+\eta(s),1}} ,
   \label{1DFCCestimate}
\end{align}
for all  $\omega \in \mathbb{R}$,
where  $\min\Big\{1,{|\omega|^{-s}}\Big\} := 1$, if $\omega  = 0$.
Moreover for $\ell = 1$ and any $p>1$ we  have
  \begin{align}\label{ell1est}
    \Big|I^{\omega} g - {I^{\omega,\ell} g}\Big| \ \leq\  {C_p} \, \min\Big\{1,|\omega|^{-1}\Big\}\,
  \left(\frac{1}{n_\ell} \right)^{p-1} \, \Vert g \Vert_{\cW^{p,1}} .
\end{align}
\end{theorem}

\begin{proof}
 When   $\ell \geq 2$, we have  $n_\ell \geq 2$. A slight variation of  \cite[Theorem 2.2]{GrahamDominguez:2011} then shows  that there exists a constant $C>0$ such that, for all $p > 1$,  the estimate
  \begin{align}
  \Big|I^{\omega} g - {I^{\omega,\ell} g}\Big| \ \leq\  {C} |\omega|^{-s}  \,
    \left(\frac{1}{n_\ell} \right)^{p-1}  \, \Vert g_c \Vert_{H^{p+ \eta(s)}}, \label{noteestimate}
\end{align}
holds,  for  $\vert \omega \vert \geq 1$, where $g_c(\theta) := g(\cos \theta) $ is the cosine transform of $g$ and $\Vert \cdot  \Vert_{H^p}$ is the usual univariate Sobolev norm of order $p$. The estimate in \cite[Theorem 2.2]{GrahamDominguez:2011} is only stated for  positive $\omega$, but  the negative case  can be trivially obtained  from this by replacing $y $ by $-y$ and $\omega$ by $-\omega$ in  \eqref{1DoscIntegralEq}.

For $\vert \omega \vert < 1$,  we proceed by estimating  the error in the  classical Clenshaw-Curtis rule applied to  $\widehat{g} (y) := g(y) \exp(\ri \omega y)$ by
\begin{align} \left\vert I^\omega g  - I^{\omega,\ell} g \right\vert & =
  \left\vert \int_{-1}^1 (I - Q^\ell) \widehat{g} \right\vert    \leq {\sqrt{2}} \Vert (I - Q^\ell ) \widehat{g} \Vert_{L^2[-1,1]}
\nonumber   \\ & \leq {\sqrt{2}}  \left(\int_0^\pi \vert ((I - Q^\ell) \widehat{g} )_c(\theta) \vert^2 \rd \theta \right)^{1/2}\nonumber  \\ & \leq C \left(\frac{1}{n_\ell}
                                                                                                                                                 \right)^{p}\Vert \widehat{g}_c\Vert_{H^p}.
                                                                                                                                                 \label{noteestimate1}
\end{align}
(In the last step,  we used  \cite[eq (15)]{GrahamDominguez:2011}).

The constants  $C$ appearing in \eqref{noteestimate} and  \eqref{noteestimate1} are  independent of $\ell$ and $\omega$ as well as $p$, but to  complete the proof we need to estimate the Sobolev norms  on the right-hand sides of \eqref{noteestimate} and \eqref{noteestimate1} in terms of the $\cW^{p,1}$ norm of $g$. This is where the $p$-dependence appears. Since all the derivatives of $\cos \theta$ are bounded above by $1$, the Faa di Bruno formula (e.g.  \cite{Ha:06})) readily yields
$$ \left\vert \left( \frac{\rd}{\rd \theta} \right)^p (g_c(\theta)) \right\vert \ \leq\ \sum_{\cP} \vert g^{(\vert \cP \vert )} (\cos \theta) \vert , \quad \theta \in [-\pi, \pi], $$
where the sum is over all partitions $\cP$ of the set $\{1, \ldots , p\}$, with $\vert \cP\vert$ denoting the number of subsets  in $\cP$. Since  $\vert \cP \vert \leq p$ and the number of all such partitions is $B_p$ (the $p$th {\em Bell number}) it follows that $\Vert g_c\Vert_{H^p} \leq C B_p \Vert g \Vert_{\cW^{p,1}}.$
\blue{Also, since $\widehat{g}_c(\theta) = g_c(\theta) \exp(\ri \omega \cos \theta)$,  an application of the Leibniz rule shows that} $$ \blue{\Vert \widehat{g}_c \Vert_{H^p} \leq K_p \Vert g_c\Vert_{H^p} \leq C B_p K_p \Vert g \Vert_{\cW^{p,1}},}$$
\blue{and combining this with \eqref{noteestimate1}, 
  completes the proof of \eqref{1DFCCestimate}.}

To obtain \eqref{ell1est} for $\vert \omega \vert   \geq 1$, we integrate by parts to obtain
\begin{align*} I^\omega g - I^{\omega,1} g & = \int_{-1}^1 (g(y) - g(0) ) \exp(\ri \omega y) \rd y \\
                                           & = \frac{1}{\ri \omega} \left( \left[ (g(y) - g(0))  \exp(\ri \omega y) \right]_{-1}^1 -
                 \int_{-1}^1 g'(y) \exp(\ri \omega y ) \rd y  \right),
\end{align*}
from which the required estimate follows, since $n_1 = 1$. When $\vert \omega \vert < 1$,
the proof is trivial.
\end{proof}

\begin{remark}[Dependence on $p$] In the proof  above,  the constant $C_p$  in \eqref{1DFCCestimate} grows quickly with $p$, in fact with the order of the growth of the Bell number $B_p$. This is the price we pay for an   estimate to be uniform in $\ell$.
(Uniformity of the estimate with respect to  $\ell$ is required in the proof of Theorem \ref{thm:main}, because in   the Smolyak construction, low order approximations in some dimensions are combined with high order in others, so we need estimates for all orders.)    If \eqref{1DFCCestimate} were only required to hold for  $\ell$ sufficiently large (relative to $p$), then the constant $C_{p}$ can be bounded independently of $p$ (see, for example, \cite[Remark 2.4]{GrahamDominguez:2013}).
\end{remark}

  \begin{remark}[\blue{Relation to results  from complexity theory}] \label{rem:Novak}
\blue{The results in Theorem \ref{prop:Errorestimate1dFCC} imply that, for $p \geq 2$,
  \begin{align} \sup_{\Vert g \Vert_{\cW^{p,1}} \leq 1} \left \vert I^{\omega}g  - I^{\omega, \ell}g \right\vert   \lesssim \vert \omega\vert^{-1} \left(\frac{1}{n_\ell}\right)^{p-1} \quad \text{and}  \quad \sup_{\Vert g \Vert_{\cW^{p+3,1}}\leq 1}
\left \vert I^{\omega}g  - I^{\omega, \ell}g \right\vert  \lesssim \vert \omega\vert^{-2} \left(\frac{1}{n_\ell}\right)^{p-1}, \label{worst_case} \end{align}
where  $\lesssim$ indicates that  a generic multiplicative constant independent of $\omega$, $p$ and $\ell$ is omitted.
These are bounds for the ``{\em worst-case error}'' as studied in complexity theory - see   \cite{NoUlWo:15}, \cite{No:23},  where the authors  develop  this theory for the one-dimensional oscillatory integral $I^\omega g $
(albeit for integer $\omega$ and in Sobolev spaces,  rather than the spaces of functions with continuous partial derivatives  considered here).
In  that literature the {\em normalized error}  is defined as  the quotient  of the absolute error
in \eqref{worst_case} and the norm of $I^\omega$. Since trivial integration by parts yields
$$ I^\omega g = \frac{1}{\ri \omega} \left( \left[\exp(\ri \omega y) g(y) \right]_{-1}^1 - \int_{-1}^{1} \exp(\ri \omega y) g'(y) \rd y \right),$$
  we have  $\Vert I^\omega \Vert_{\cW^{p,1}} \lesssim \vert \omega \vert^{-1}  $, for all $p\geq 1$.  Moreover,  this bound is sharp in terms of $\omega$ dependence.   In particular:
$$\mathrm{(a)} \quad I^\omega 1 = \frac{2}{\omega} \sin \omega \quad \text{and} \quad \mathrm{(b)} \quad I^\omega y =
\frac{2}{\ri \omega}\left(\cos \omega -  \frac{\sin \omega}{\omega}\right). $$ Using (a) for $\omega$ bounded away from integer multiples of $\pi$ and (b) otherwise, we obtain   $ \vert \omega \vert^{-1} \lesssim \Vert I^\omega \Vert_{\cW^{p,1}}   $, for all $p \geq 0$. Thus \eqref{worst_case} can be rewritten:
 \begin{align} \sup_{\Vert g \Vert_{\cW^{p,1}} \leq 1} \left \vert I^{\omega}g  - I^{\omega, \ell}g \right\vert   \lesssim \left(\frac{1}{n_\ell}\right)^{p-1} \Vert I^\omega \Vert_{\cW^{p,1}}\quad \text{and}  \quad \sup_{\Vert g \Vert_{\cW^{p+3,1}}\leq 1}
   \left \vert I^{\omega}g  - I^{\omega, \ell}g \right\vert  \lesssim \vert \omega\vert^{-1} \left(\frac{1}{n_\ell}\right)^{p-1} \Vert I^\omega \Vert_{\cW^{p+3,1}}, \label{worst_case_norm} \end{align}
 yielding bounds on the normalized error (as defined in \cite{NoUlWo:15}, \cite{No:23})}. \blue{While both bounds in \eqref{worst_case} show that as $\vert \omega \vert \rightarrow \infty $, the (absolute) error gets smaller and in some sense the problem becomes `easier',
 the second bound in \eqref{worst_case_norm} shows that the same is true for the normalized error, provided we are working in spaces of sufficiently smooth functions.}

\blue{Compared to these results, the detailed analysis in \cite{NoUlWo:15} has some  major differences:
  While \eqref{worst_case}  gives error estimates for our  specific FCC rule, \cite{NoUlWo:15} studies
  the infimum
of this worst case error (absolute or  normalized) over all quadrature rules using a given number of point evaluations of $f$  (thus including the FCC rule as a special case). Moreover \cite{NoUlWo:15} gives both upper and lower error bounds (essential in complexity theory). Neverthess  \cite{NoUlWo:15} contains some qualitative  statements which are comparable  to ours, in particular (a) When considering absolute errors, the problem becomes easier for large $\vert \omega\vert$ (\cite[Corollary 16]{NoUlWo:15}); (b) When considering normalized errors the problem also becomes easier for large $\vert\omega\vert$, but this is only true  when $g$ has enough  Sobolev regularity \cite[Corollary 18]{NoUlWo:15}.  }

\blue{While the results in \cite{NoUlWo:15} only apply in  one dimension,
  the purpose of this paper is to study the high-dimensional FCCS rule.
  Here we show that our method maintains the aforementioned advantages for computing $\cI^{k,d,\ba} f$:  as $|k|$ grows,  the accuracy of our method improves when considering the  absolute error criterion; see Theorem \ref{thm:main_grid}, and also improves when considering normalized error, under certain mild additional  regularity conditions; see Theorem \ref{thm:main_grid1}. 
}

\end{remark}

\section{An FCCS rule for multi-dimensional integrals}
\label{sec:FCCS_high_dimension}
\noindent
The direct application of the tensor product version of a  conventional 1D rule  to the multi-dimensional problem  \eqref{highDoscIntegralEq} will give very poor results as $d$ or $k$  increases, \blue{firstly} because of  the high oscillation and \blue{secondly}  because of the curse of dimensionality.  The difficulty is illustrated by the following simple example.
  \begin{example} \label{prop:bad}
  	Suppose the integral \eqref{highDoscIntegralEq} is approximated  by the tensor product of the 1d  Clenshaw-Curtis rule \eqref{CC},   using  $n+1$ points in each coordinate direction, so that the integrand is evaluated at  $N := (n+1)^d$ points. Then,  in the special case $f(\by) = 1$ and $\ba = (1, 0, 0, \ldots, 0)^\top$,
    the error is
    $$
     \int_{[-1,1]^{d-1}} \int_{-1}^1 (I - P^n) \bigg( \exp(\ri k \cdot)\bigg)(y_1)  \,   \rd y_1  \rd y_2 \ldots \rd y_d       \ =\  2^{d-1} \, \int_{-1}^1 (I - P^n)\bigg(\exp(\ri k \cdot)\bigg)(y_1)  \,   \rd y_1 ,   $$
       where $P^n$ denotes the polynomial interpolant at $n+1$ Clenshaw-Curtis points in 1d (i.e. the operator $Q^\ell$ in \eqref{ChebyshevExpansion1D},  with $n_\ell$ replaced by $n$).
The error estimate for this is
     $k^{p}\,  n^{-p} = k^p \, \mathcal{O}(N^{-p/d})$ for any $p$.
   \end{example}
We alleviate  the problem of growth with respect to $k$ by adopting   the Filon approach described above.
Then, to   reduce the effect of dimension (encapsulated in the $N^{-p/d}$ term), we approximate
 \eqref{highDoscIntegralEq} by replacing $\hf$ in \eqref{highDoscIntegralEq1} by  its Smolyak {interpolant} $\cQ^{r,d} \hf$ defined as follows.


Using the 1D interpolation operator  $Q^\ell$  in \eqref{1Dinterpolation-operator},
and  the nested sequence of Clenshaw-Curtis grids 
in  \eqref{eq:grids},    we define the  difference  operator $D^{\ell} $ (for $\ell \geq 1$)  by
\begin{align}
{D}^{\ell}g:=(Q^{\ell}-Q^{\ell-1})g,
\quad \text{with} \quad  Q^{0} g :=0.
\label{FCCSmolyak1D_Diff}
\end{align}
To define the Smolyak interpolation operator, it is convenient to define the index set
$\Lambda(q,d)$, for  integers $q, d$ with  $q \geq d$ by
$$ \Lambda (q,d) \ =\  \{ \bell \in \mathbb{N}^d :  \boldsymbol{1} \leq \bell, \,  \vert \bell \vert \leq q\},  $$
where  $\boldsymbol{1} = (1,\ldots, 1)^\top$ and $\vert \bell \vert = \ell_1 + \ell_2 + \ldots + \ell_d$. By \cite[p.13]{WaWo:95}, the cardinality of $\Lambda(q,d)$ is given by the binomial coefficient:
\begin{align} \label{card} \# \Lambda(q,d) = \left(\begin{array}{l}q\\d\end{array}\right). \end{align}
Smolyak's formula for interpolating any  function $f: \mathbb{R}^d \rightarrow \mathbb{R}$,  with maximum level $r\in \mathbb{N} := \{1,2,3, \ldots\}$, is then given (e.g., in \cite[eq (10)]{WaWo:95} or  \cite[p.214]{GeGr:98}) by
\begin{align}
(\mathcal{Q}^{r,d}f)(\by):= \sum_{\boldsymbol{\ell} \in \Lambda(r + d -1,d)}({D}^{\ell_1}\otimes...\otimes {D}^{\ell_d})f (\by)  .
\label{FCCSmolyak-dD}
\end{align}
See \cite[Prop 6]{BaNoRi:00} for a discussion of the interpolatory properties of $\mathcal{Q}^{r,d}$. The notation $ {D}^{\ell_1}\otimes...\otimes {D}^{\ell_d} $ indicates that we apply $ {D}^{\ell_j} $ with respect to variable $ y_j $, for each $ j=1,...,d $. Note that,  since $\boldsymbol{\ell} \in \Lambda(r+d-1,d)$ in \eqref{FCCSmolyak-dD},    we have
  $ d \leq \vert \bell\vert \leq r+d-1$ and   also $1 \leq \ell_j \leq r$ for each $j = 1, \ldots, d$.

Then, to define the FCCS rule for \eqref{highDoscIntegralEq1} (and hence     
  \eqref{highDoscIntegralEq}),
we replace $\hf$   in \eqref{highDoscIntegralEq1} by 
$(\mathcal{Q}^{r,d}\hf)$, thus obtaining  
\eqref{FCCSmolyak_intro}.
An alternative formula for $\mathcal{Q}^{r,d} \hf$ is obtained  using the {\em combination technique}  (\cite[Lemma 1]{WaWo:95}, \cite[Section 4.1]{GeGr:98}). This allows the formula \eqref{FCCSmolyak-dD} to  be written in terms of $Q^{l}$ instead of ${D}^{l}$; the result is:
\begin{align}
(\mathcal{Q}^{r,d}\hf)(\by):= \sum_{\stackrel{\boldsymbol{\ell} \geq \boldsymbol{1}}{r\leq |\boldsymbol{\ell}| \leq r+d-1}}\, (-1)^{r+d-|\boldsymbol{\ell}|-1}\binom{d-1}{|\boldsymbol{\ell}|-r}({Q}^{\ell_1}\otimes...\otimes {Q}^{\ell_d})\hf (\by). \label{FCCSmolyak-dD-alt}
\end{align}

A useful observation from this is (when $d=1$),
  \begin{align} \label{eq:34a} \mathcal{Q}^{r,1} \widehat{f} = Q^r \widehat{f} \end{align}
Then, inserting     \eqref{FCCSmolyak-dD-alt} into   \eqref{highDoscIntegralEq1},
we obtain the following approximation of \eqref{highDoscIntegralEq}:

\begin{proposition}\label{prop:FCCS_alt}
\begin{align}
\mathcal{I}^{k,d,\ba,r} f
  &=\sum_{\stackrel{\boldsymbol{\ell} \geq \boldsymbol{1}}{r\leq |\boldsymbol{\ell}| \leq r+d-1}}
(-1)^{r+d-|\boldsymbol{\ell}|-1}\binom{d-1}{|\boldsymbol{\ell}|-r}
(I^{\omega_1,\ell_1}\otimes\cdot\cdot\cdot\otimes I^{\omega_d,\ell_d})f ,
\label{FCCSmolyak_highDoscIntegralEq2}
\end{align}
where \begin{align} \label{scale}
        \omega_j=ka_j \quad \text{for each} \quad  j=1,...,d.\end{align}
    \end{proposition}
\begin{proof}
By \eqref{FCCSmolyak_intro} and \eqref{FCCSmolyak-dD-alt}, it is sufficient to prove that 
\begin{align}
\int_{[-1, 1]^d} ({Q}^{\ell_1}\otimes...\otimes {Q}^{\ell_d})\hf (\by) \exp(\ri k \tba \cdot \by) \rd \by = (I^{\omega_1,\ell_1}\otimes\cdot\cdot\cdot\otimes I^{\omega_d,\ell_d})f,
\label{highD_Q_to_I}
\end{align}
with $\omega_j$ as given in \eqref{scale}.
The proof of \eqref{highD_Q_to_I} is obtained  by induction on the dimension $d$. For $d = 1$,
we have, directly from Proposition \ref{prop:simple},
\begin{align}
\int_{-1}^1 (Q^{\ell_1} \hf)(y) \exp(\ri  k \ta_1 y  ) \rd y   = I^{ka_1, \ell_1}f =  I^{\omega_1, \ell_1}f.
\end{align}

Now suppose \eqref{highD_Q_to_I} holds for dimension  $d$ and, considering  dimension $d+1$, we  introduce new notation as follows.
For  $\by \in [-1,1]^{d + 1}$ and $\ba \in \mathbb{R}^{d + 1}$, we write   $\by = (\by^*, y_{d + 1})$ and $\ba = (\ba^*, a_{d + 1})$, where  $\by^* \in [-1, 1]^d$ and $\ba^* \in \mathbb{R}^d$.
Moreover,  for any $f\in C( [-1, 1]^{d + 1})$  and any fixed $\by^* \in [-1, 1]^d$, we write $f(\by^*, \cdot)$ to denote the univariate function $y_{d + 1} \mapsto f(\by^*, y_{d + 1}) = f(\by)$. 
Using $I^{\omega, \ell}[f(\by^*, \cdot)]$ to denote the application of the quadrature rule \eqref{eq:def:quad} to   $f(\by^*, \cdot)$,  we also define the $d-$variate functions
\begin{align}\label{dvariate} F^{\omega, \ell}(\by^*) =   I^{\omega, \ell}[f(\by^*, \cdot)], \quad \text{and} \ \widehat{F}^{\omega, \ell}(\by^*) = F^{\omega, \ell}(\by^*) \exp(\ri k \hba^* \cdot \by^*)  \quad \text{for all} \ \by^* \in [-1,1]^d.
\end{align}
Then, we have
\begin{align}
& \int_{[-1, 1]^{d + 1}} ({Q}^{\ell_1}\otimes...\otimes {Q}^{\ell_d} \otimes {Q}^{\ell_{d + 1}})\hf (\by) \exp(\ri k \tba \cdot \by) \rd \by \nonumber \\
& \mbox{\hspace{0.5cm}} = \int_{[-1, 1]^d} ({Q}^{\ell_1}\otimes...\otimes {Q}^{\ell_d}) \left( \int_{-1}^1 ({Q}^{\ell_{d + 1}} \hf (\by^*, \cdot) \exp(\ri k \ta_{d + 1} \cdot)) (y_{d + 1}) \rd y_{d + 1} \right) \exp(\ri k \tba^* \cdot \by^* ) \rd \by^*\label{eq:311}
\end{align}
By  Proposition \ref{prop:simple},  we have
$$\int_{-1}^1 ({Q}^{\ell_{d + 1}} \hf (\by^*, \cdot) \exp(\ri k \ta_{d + 1} \cdot)) (y_{d + 1}) \rd y_{d + 1}  = I^{\omega_{d+1}, \ell_{d+1}} \left[f(\by^*,\cdot)\right] \exp(\ri k \hba^*\cdot \by^*)  = \widehat{F}^{\omega_{d+1}, \ell_{d+1}} (\by^*).   $$
Inserting this into \eqref{eq:311} and  using the inductive hypothesis (i.e., that \eqref{highD_Q_to_I} holds), we obtain
\begin{align}
  & \int_{[-1, 1]^{d + 1}} ({Q}^{\ell_1}\otimes...\otimes {Q}^{\ell_d} \otimes {Q}^{\ell_{d + 1}})\hf (\by) \exp(\ri k \tba \cdot \by) \rd \by \nonumber \\
  & \mbox{\hspace{1in}} = \int_{[-1, 1]^{d}} \left(({Q}^{\ell_1}\otimes...\otimes {Q}^{\ell_d})\widehat{F}^{\omega_{d+1}, \ell_{d+1} } \right) (\by^*) \exp(\ri k \tba^* \cdot \by^*) \rd \by^* \nonumber \\
  & \mbox{\hspace{1.5in}} = \left(I^{\omega_1, \ell_1} \otimes \ldots \otimes I^{\omega_d, \ell_d}\right) \,  F^{\omega_{d+1}, \ell_{d+1}}.\label{insert}
\end{align}

The fact that \eqref{highD_Q_to_I} holds for dimension $d+1$ then  follows by the defintion of $F^{\omega_{d+1},\ell_{d+1}}$ in \eqref{dvariate}.
\end{proof}

\section{Error analysis of the FCCS  rule}
\label{sec:Error_Analysis}
\noindent
In this section, we shall provide an error estimate of the FCCS rule \eqref{FCCSmolyak_highDoscIntegralEq2}  for approximating  \eqref{highDoscIntegralEq} (equivalently \eqref{highDoscIntegralEq1}).
Before this we need several preliminary results.

\begin{lemma} \label{lem:1}   (i) Recall $I^\omega g$ defined in \eqref{1DoscIntegralEq}. Then,
 for any $\omega \in \mathbb{R}$ and $g \in \cW^{1,1}$,
 $$ \vert I^{\omega} g \vert \  \leq \ 4 \min\{1,|\omega|^{-1}\}\,  \Vert g \Vert_{\cW^{1,1}},$$
 where $ \min\{1,|\omega|^{-1}\} :=1 $ when $\omega = 0$ .\\
  (ii) Let  $g \in \cW^{p,1}$,  with  $p > 1$, let  $a \in \mathbb{R}$, and  define $\widetilde{a}, \ha$ as   {in Notation \ref{not:tilde} and set $\widehat{g}(y) = g(y) \exp(\ri k \ha y)$. } Then, for all  $k \ge 0$,
   and  $\ell \geq 1$, we have
  \begin{align}  \left\vert \int_{-1}^1 (D^\ell \hg) (y) \exp(\ri k \ta y) \rd y \right\vert \ \leq \ C_p^{\prime} \min\{ 1, |\omega|^{-1} \} \,  \left(\frac{1}{n_{\ell-1}}\right)^{p-1} \, \Vert g \Vert_{\cW^{p,1}} , \   \quad
    \label{1dgeq1}\end{align}  where $\omega = k a$,  $C_p^{\prime} = {2 \max\{ C_p, 2 \}}$, with $C_p$ as in Proposition \ref{prop:Errorestimate1dFCC} and we have set $n_0 = 1$.  

Moreover,  for  $k \geq0$ and $\ell \geq 3$,
\begin{align}
	\left\vert \int_{-1}^1 (D^\ell \hg) (y) \exp(\ri k \ta y) \rd y \right\vert \ \leq \ C_p^{\prime} \min\{ 1, |\omega|^{-2} \} \,  \left(\frac{1}{n_{\ell-1}}\right)^{p-1} \, \Vert g \Vert_{\cW^{p+3,1}}.
\label{1dgeq3}
\end{align}
\end{lemma}
\begin{proof}
(i) For $\omega \not = 0$. we use  integration by parts to obtain
\begin{align*}
I^{\omega} g   \ = \  \int_{-1}^{1} g(y) \exp(\ri \omega y)\rd y
\ =\ \frac{1}{\ri \omega}[g(y)\exp(\ri \omega y)]_{-1}^{1} - \frac{1}{\ri \omega}\int_{-1}^{1} g'(y)\exp(\ri \omega y)\rd y.
\end{align*}
Thus
\begin{align*}
  |I^{\omega} g |  \ \leq\  \frac{2}{|\omega|}||g||_{\infty} + \frac{2}{|\omega|}||g'||_{\infty} \ \leq\  \frac{4}{|\omega|} \, ||g||_{\cW^{1,1}}.
\end{align*}
On the other hand, a direct estimate yields $ \vert I^{\omega} g \vert \leq 2 \Vert g \Vert_{\infty}$ for all $\omega$, and part (i) follows.

(ii)  For $\ell \ge 2$, by the definition of $D^\ell$,  Proposition \ref{prop:simple} and  Theorem \ref{prop:Errorestimate1dFCC} (with $s=1$),
\begin{align}
  \left\vert \int_{-1}^1 (D^\ell \hg) (y) \exp(\ri k \ta y) \rd y\right\vert  \ & = \left\vert \int_{-1}^1  (Q^\ell - Q^{\ell-1}) \hg (y) \exp(\ri k \ta y) \rd y \right\vert \nonumber \\
                                                                          & \  = \
                                                                            |{I^{\omega,\ell}g}-{I^{\omega,\ell-1}g}| \ \leq \  |I^{\omega}g-{I^{\omega,\ell}g}|+|I^{\omega}g-{I^{\omega,\ell-1}g}| \label{interim} \\
                                                                                &\leq C_p\min\{1,|\omega|^{-1}\} \left[\left(\frac{1}{n_\ell}\right)^{p-1}  +
                                                                                  \left(\frac{1}{n_{\ell-1}}\right)^{p-1}\right]\, \Vert g\Vert _{\cW^{p,1}}\nonumber \\
  &= C_p\min\{1,|\omega|^{-1}\} \left[
                  \left(\frac{1}{n_{\ell-1}}\right)^{p-1} (2^{1-p} + 1)\right]\, \Vert g\Vert _{\cW^{p,1}} \nonumber \\
&\leq 2 C_p\min\{1,|\omega|^{-1} \}  \left(\frac{1}{n_{\ell-1}}\right)^{p-1} \Vert g\Vert _{\cW^{p,1}}.
\end{align}
For $\ell = 1$, we use  Proposition \ref{prop:simple} to obtain, 
\begin{equation*}
  \left\vert \int_{-1}^1 (D^1 \hg) (y) \exp(\ri k \ta y) \rd y\right\vert  = |{I^{\omega,1}g}|\  = \left\{ \begin{array}{ll}  | g(0) | \left\vert \int_{-1}^1 \exp(\ri \omega y) dy \right\vert & \text{for} \quad  |\omega| \ge 1, \\
2 |g(0)|,  &  \text{for} \quad  |\omega| < 1, \end{array} \right.
\end{equation*}
which gives $ \left\vert \int_{-1}^1 (D^1 \hg) (y) \exp(\ri k \ta y) \rd y\right\vert \le 2 \min\{1, |\omega|^{-1}\} || g ||_{\infty} $. The last two estimates yield \eqref{1dgeq1}.


 To obtain \eqref{1dgeq3} we proceed as in the case $\ell \geq 2$, but when  estimating \eqref{interim} we can  use  Theorem \ref{prop:Errorestimate1dFCC} with $s=2$ instead of $s=1$.
\end{proof}

\begin{theorem} \label{thm:main}
Let $C_p'$ be as in Lemma \ref{lem:1} (ii)
and,  for $\ba \in \bR^d$, set  $\omega_j = k a_j$. Then
for fixed $p >  1$, $r \geq 1$, and $f \in \cW^{p,d}$,
\begin{align}
\label{R}
\vert \cI^{k,d,\ba} f - \cI^{k,d,\ba,r} f \vert \ \leq \
   \prod _{j=1}^d
  \left(C_p' \min\{ 1, | \omega_j |^{-1} \} \right)\, \left(\begin{array}{c} r+d-1\\ d-1 \end{array} \right)\,  2 ^{-(r-d) (p-1)} \Vert f \Vert_{\cW^{p,d}},
\end{align}
where $\min\{ 1, |\omega_j|^{-1}\}:= 1 $ when $\omega_j = 0$.
\end{theorem}

\begin{proof}

This is proved by induction
on  $d$.  The argument follows that of \cite{WaWo:95} and \cite{GeGr:98}, although these references analysed  tensor versions of standard quadrature rules and not Filon rules for oscillatory integrals, as considered here.

Note that for $d=1$, we have, by  \eqref{highDoscIntegralEq} and \eqref{1DoscIntegralEq},
$ \cI^{k,1,\ba} f = I^{\omega_1} f$
with $\omega_1 = ka_1 $.
By \eqref{FCCSmolyak_intro} with $d = 1$, then \eqref{eq:34a} and    Proposition \ref{prop:simple},   we have
$ \cI^{k,1,\ba,r} f \ = \ I^{\omega_1,r} f . $
Then applying Theorem  \ref{prop:Errorestimate1dFCC} with $s = 1$, we have
\begin{align} \label{1d} \vert \cI^{k,1,\ba} f -  \cI^{k,1,\ba, r} f \vert \ = \  \vert I^{\omega_1} f -  I^{\omega_1,r} f \vert \ \leq \ C_p \min\{1,|\omega_1|^{-1} \} \left( \frac{1}{n_r} \right)^{p-1} \Vert f \Vert_{\cW^{p,1}}.  \end{align}
Since $C_p \leq C_p'$, this yields   \eqref{R} for $d = 1$.


Now suppose \eqref{R} holds for  dimension $d$ and consider dimension $d+1$.
We adopt the notation used in the proof of Proposition \ref{prop:FCCS_alt}, and,  for $\bell \in \bN^{d+1}$, we write $\bell = (\bell^*, \ell_{d+1})$ with $\bell^* \in \bN^d$. Also, recalling \eqref{1DoscIntegralEq}, we
introduce, for $f \in C([-1,1]^{d+1})$,
\begin{align} F^{\omega} (\by^*) := I^{\omega} [f(\by^*, \cdot)]. \label{Fomega}
\end{align}
We then  estimate  the error by the sum of two terms:
\begin{align}
\vert \cI^{k,d+1,\ba} f - \cI^{k,d+1,\ba,r} f \vert & \ \le \ \vert \cI^{k,d+1,\ba} f - \cI^{k,d,\ba^*,r} F^{\omega_{d+1}} \vert
+ \vert \cI^{k,d,\ba^*,r} F^{\omega_{d+1}}  - \cI^{k,d+1,\ba,r} f \vert \nonumber \\
& \ =: \ \vert T_1 \vert + \vert T_2 \vert.
\label{E1}
\end{align}

Considering  $\vert T_1 \vert$ first, we  use  the equality $\cI^{k,d+1,\ba} f = \cI^{k,d, \ba^*} {F^{\omega_{d+1}}} $ and the inductive hypothesis  to obtain
\begin{align}
\vert T_1\vert \ & = \left\vert \left(\cI^{k,d,\ba^*} - \cI^{k,d,\ba^*,r} \right) {F^{\omega_{d+1} }}  \right\vert \nonumber \\
& \leq \ 
                                                                                                                              \left(\begin{array}{c} r+d-1\\ d-1 \end{array} \right)
\prod_{j=1}^d
  \left({C_p'} \min\{1, |\omega_j|^{-1} \} \right)\, 2^{-(r-d)(p-1)} \Vert F^{\omega_{d+1}} \Vert_{\cW^{p,d}}.
\label{eq:Idag}   
\end{align}
Then, using Lemma \ref{lem:1} (i) and a little manipulation one can see that
\begin{align} \label{manip}\Vert F^{\omega_{d+1}} \Vert_{W^{p,d}} \ \leq \ 4 \min\{ 1, \vert \omega_{d+1}\vert^{-1}\}  \, \Vert f \Vert_{W^{p,d+1}} \leq \ {C^{\prime}_p} \min\{ 1, \vert \omega_{d+1}\vert^{-1}\}  \, \Vert f \Vert_{W^{p,d+1}}. \end{align}
Hence, using this in \eqref{eq:Idag} we obtain
\begin{align}
\vert T_1\vert \ 
& \leq  \left(\begin{array}{c} r+d-1\\ d-1 \end{array} \right) \, \prod_{j=1}^{d+1}
\left(C_p' \min\{1, |\omega_j|^{-1} \} \right)\, 2^{-(r-d)(p-1)} \Vert f \Vert_{\cW^{p,d+1}}.
\label{G}
\end{align}

To estimate  $T_2$, we first apply  \cite[eq. (11)]{WaWo:95}, to write
\begin{align*} 
\cQ^{r,d+1} \ &=\  \sum_{\bell \in \Lambda(r+d,d+1)}  \ ({D}^{\ell_1}\otimes...\otimes {D}^{\ell_{d+1}})
\                              = \ \sum_{\bell^* \in \Lambda(r+d-1,d)}  \ ({D}^{\ell^*_1}\otimes...\otimes {D}^{\ell^*_{d}}) \otimes Q^{r + d - \vert \bell^* \vert} .
\end{align*}
Then, {with notation as in \eqref{dvariate}} and proceeding as in \eqref{insert},  we have
\begin{align}
 \cI^{k,d+1, \ba,r} f
&=  \int_{[-1,1]^{d+1}}  (\cQ^{r,d+1} \hf)(\by) \exp(\ri k \tba \cdot \by) \rd \by \nonumber \\
& \mbox{\hspace{-1.5cm}} =  \sum_{\bell^* \in \Lambda(r+d-1,d)} \int_{[-1,1]^d}   \ ({D}^{\ell^*_1}\otimes...\otimes {D}^{\ell^*_{d}}) \widehat{F}^{ \omega_{d+1}, r + d - \vert \bell^* \vert } (\by^*) \exp(\ri k \tba^* \cdot \by^*) \rd \by^*.  \label{F1}
\end{align}
 Also, by \eqref{FCCSmolyak_intro} and subsequently  \eqref{FCCSmolyak-dD},   we have
  \begin{align}
    \cI^{k,d,\ba^*,r} F^{\omega_{d+1}} & = \int_{[-1,1]^d}({\cQ^{r,d}} \widehat{F}^{\omega_{d+1}})(\by^*) \exp(\ri k \widetilde{\ba}^*\cdot \by^*) \rd \by^*\nonumber  \\ & = \sum_{\bell^* \in \Lambda(r+d-1,d)} \int_{[-1,1]^d}   \ ({D}^{\ell^*_1}\otimes...\otimes {D}^{\ell^*_{d}}) \widehat{F}^{ \omega_{d+1}} (\by^*) \exp(\ri k \tba^* \cdot \by^*) \rd \by^*.
                                                                                                                                                                        \label{F1a} \end{align}
                                                                                                                                                                      Thus,  combining \eqref{F1} and \eqref{F1a}, we obtain
\begin{align}
T_2 & = \cI^{k,d,\ba^*,r} F^{\omega_{d+1}} - \cI^{k,d+1,\ba,r} f \nonumber \\
& \mbox{\hspace{-0.8cm}} =  \sum_{\bell^* \in \Lambda(r+d-1,d)}  \int_{[-1,1]^d}  ({D}^{\ell^*_1}\otimes...\otimes {D}^{\ell^*_{d}})  \left( \widehat{F}^{\omega_{d + 1}} - \widehat{F}^{\omega_{d + 1}, r + d - \vert \bell^* \vert }\right) (\by^*)  \exp(\ri k \tba^* \cdot \by^*) \rd \by^*.  
 \label{E0}
\end{align}

Now, to  estimate $\vert T_2 \vert$, we consider any function  {$G \in \cW^{p,d}$,  any $\ba \in \mathbb{R}^d$ and $k > 0$, and define  $\widehat{G}(\by) = G(\by) \exp(\ri k \hba. \by) $, with $\hba$ as defined in Notation \ref{not:tilde} as in
  \eqref{highDoscIntegralEq1}.}
By induction on dimension $d$,  using Lemma \ref{lem:1} (ii) at each step,
we obtain   the estimate
\begin{align}
  & \left \vert \int_{[-1,1]^d}  ({D}^{\ell^*_1}\otimes...\otimes {D}^{\ell^*_{d}})  \,    \widehat{G} (\by^*)  \exp(\ri k \tba^* \cdot \by^*) \rd \by^* \right\vert \nonumber \\
  &   \mbox{\hspace{3cm}} \leq \ \prod_{j=1}^d  \left( {C_p'} \min\{ 1, \vert \omega_j\vert^{-1} \} \left( \frac{1}{n_{\ell_j^*-1}} \right)^{p-1} \right) \Vert G \Vert_{\cW^{p,d}}      \nonumber \\
  & \mbox{\hspace{3.5cm}} \leq  \   \left[\prod_{j=1}^d
    \left( C_p' \min\{ 1, \vert \omega_j\vert^{-1} \}  \right) \right]  2^{(2d -\vert \bell^*\vert )(p-1)}\,   \Vert G \Vert_{\cW^{p,d}} , \label{induction} \end{align}
where,  in the last step,  we used the fact that
$$
\prod_{j=1}^d \frac{1}{n_{\ell_j^* - 1}} \leq \prod_{j=1}^d 2^{2-\ell_j^*} = 2^{2d - \vert \bell^*\vert }.
$$

Thus, to complete the estimate of  \eqref{E0}  we need to estimate $\Vert F^{\omega_{d+1}} - F^{\omega_{d+1}, r + d - \vert \bell^*\vert} \Vert_{\cW^{p,d}}$. Recalling that $F^{\omega}$ is given by \eqref{Fomega} and $F^{\omega,\ell}$ is given in \eqref{dvariate}, we apply   Theorem  \ref{prop:Errorestimate1dFCC} with $s = 1$ to   obtain, for any $j = 1, \ldots, d$, and any $q = 0, \ldots, p$,
\begin{align*} 
     \left\vert \partial_j^q ({F^{\omega_{d+1}} - F^{\omega_{d+1}, r + d - \vert \bell^*\vert}})(\by^*) \right\vert &=
           \left\vert \left(I^{\omega_{d+1}} - I^{\omega_{d+1}, r + d - \vert \bell^* \vert} \right)\left[(\partial_j^q f)(\by^*, \cdot)\right]  \right\vert \nonumber \\
&   \mbox{\hspace{-2cm}}\leq C_p \min\{1, |\omega_{d+1}|^{-1} \}
    \, 2^{-(r+d-\vert \bell^* \vert - 1 ) (p-1)}\,  \Vert (\partial_j^q f)(\by^*, \cdot) \Vert_{\cW^{p,1}} ,
\end{align*}
 {(where we used the fact  that $r+d-\vert \bell^*\vert \geq 1$ in \eqref{E0}, because $\bell^* \in \Lambda(r+d-1,d)$)}. Hence    \begin{align} \label{E4}
        \Vert F^{\omega_{d+1}} - F^{\omega_{d+1}, r + d - \vert \bell^*\vert} \Vert_{\cW^{p,d}} \leq \ C_p \min\{1, |\omega_{d+1}|^{-1} \}
       \,  2^{-(r+d-\vert \bell^* \vert - 1 ) (p-1)} \, \Vert f \Vert_{\cW^{p,d+1}} .
        \end{align}
Combining \eqref{E0}, \eqref{induction} and  \eqref{E4}, using $C_p \leq C_p'$  and then   the cardinality formula  \eqref{card}, we obtain
\begin{align}
\vert T_2 \vert & \leq \ \left[\prod_{j=1}^{d+1}
\left(C_p^{\prime} \min\{1, |\omega_j|^{-1} \} \right) \right] \, \left(\sum_{\bell^* \in \Lambda(r+d-1,d)} \left(2^{2d -\vert \bell^* \vert } 2^{-(r+d- \vert \bell^* \vert - 1 )} \right)^{p-1} \right) \, \Vert f \Vert_{\cW^{p,d+1}}\nonumber \\
& = \left(\begin{array}{c} r+d-1\\ d \end{array} \right) \left[\prod_{j=1}^{d+1}
\left(C_p^{\prime} \min\{1, |\omega_j|^{-1} \} \right)\right] \,  2^{-(r - d - 1)(p-1)} \, \Vert f \Vert_{\cW^{p,d+1}}.
\label{E4a}
\end{align}

Then,  combining \eqref{G}, \eqref{E4a} and using  the elementary identity
$ \binom{q}{d}+\binom{q}{d-1}=\binom{q+1}{d} $, with $q=r+d-1$, we have shown that
the estimate \eqref{R} holds for dimension $d+1$.
\end{proof}

\begin{remark} \label{rem:highernegative}
  \begin{itemize} \item[(i)]   For those values of $\bell^*$ satisfying  $r + d - \vert \bell^*\vert \geq 2$, {an application of Theorem \ref{prop:Errorestimate1dFCC} shows that  \eqref{E4} holds with $\vert \omega_{d+1}\vert^{-1}$ replaced by $\vert \omega_{d+1}\vert^{-2}$ and $\cW^{p,d+1}$ replaced by $\cW^{p+3,d+1}$}.
  \item[(ii)] If any $\ell_j^* \geq 3$ then {(for that particular $j$),  an application of Theorem \ref{lem:1} (ii),  shows that   $\vert \omega_{j}\vert^{-1}$ can be replaced by $\vert \omega_{j}\vert^{-2}$ and $\cW^{p,d}$ replaced by $\cW^{p+3,d}$ in \eqref{induction}}.
    \end{itemize}
  \end{remark}

While Theorem \ref{thm:main} gives an estimate for the error which is explicit in $\ba$ and $k$, the following theorem gives a simpler (and higher order in $k^{-1}$)  at the cost of   a  stronger regularity requirement.
\begin{theorem}\label{thm:main_expk}
   Let    $\ba \in \mathbb{R}^d$ with $\ba_j \not = 0$ for each $j$. Let   $p \geq 1 $ and $d \geq 1$. Then there exists a constant $C_{p,d,\ba}$ such that, for all {$r \geq d + 1 $ and $f \in \cW^{p+3,d}$},
\begin{align}
\label{R1}
\vert \cI^{k,d,\ba} f - \cI^{k,d,\ba,r} f \vert \ \leq \
 C_{p,d,\ba}  \, \ k^{-(d + 1)} \, \left(\begin{array}{c} r+d-1\\ d-1 \end{array} \right)  \, 2 ^{-(r-d) (p-1)} \Vert f \Vert_{\cW^{p+3,d}}.
\end{align}

\end{theorem}
\begin{proof}
  We follow  the  proof of Theorem \ref{thm:main}, indicating differences only briefly.
  This proof  is less technical since 
  we do not
  require explicitness with respect to $\ba$. Throughout the proof $C_{p,d,\ba}$ denotes a generic constant which may depend on $p,d,\ba$, {and whose value may vary from line to line}.

  \bigskip

\noindent{\bf Step 1} \   For the  case  $d = 1$, we have  $r \geq 2$. Then we  recall \eqref{1d},  but this time we use Theorem \ref{prop:Errorestimate1dFCC} with $s = 2$  to obtain    $$ \vert \cI^{k,1,\ba} f -  \cI^{k,1,\ba, r} f \vert \ = \  \vert I^{\omega_1} f -  I^{\omega_1,r} f \vert \ \leq \ C_p (k\vert a \vert)^{-2}  \left( \frac{1}{n_r} \right)^{p-1} \Vert f \Vert_{\cW^{p+3,1}},   $$
which yields    \eqref{R1} for $d = 1$.

\bigskip

\noindent
{\bf Step 2} \ Now assuming \eqref{R1}  holds for $d$ we consider  the corresponding result for dimension $d+1$. In this case we are assuming
\begin{align} \label{assume}
	r \geq d+2.
\end{align}
Again we introduce the splitting  \eqref{E1}.

                                                                                        \bigskip
                                                                                        \noindent {\bf Step 2a}\
Analogously to  \eqref{eq:Idag} and \eqref{manip}  we obtain, via the inductive hypothesis,
\begin{align}
  \vert T_1\vert \
&   \leq \ 
                                                                                           C_{p,d,\ba}\,                      k^{-(d+1)} \,                \left(\begin{array}{c} r+d-1\\ d-1 \end{array} \right)
\, 2^{-(r-d)(p-1)} \Vert F^{\omega_{d+1}} \Vert_{\cW^{p+3,d}} \nonumber \\ & \leq \ C_{p,d,\ba}  \, k^{-(d+2)} \, \left(\begin{array}{c} r+d-1\\ d-1 \end{array} \right)  2^{-(r-d)(p-1)} \Vert f \Vert_{\cW^{p+3,d+1}},
\label{eq:Idag1}
\end{align}
where the additional power of $k$ comes from the estimate of $\Vert F^{\omega_{d+1}}\Vert_{\cW^{p+3, d+1}}$ -- analgous to  \eqref{manip}.

\noindent {\bf Step 2b}\
The estimate of  $\vert T_2\vert$ (starting from    \eqref{E0}) is slightly more complicated.
Note first that \eqref{E0}  is a sum of terms, each  corresponding to a different choice of $\bell^*$. In all cases $r+d-\vert \bell^*\vert \geq 1$. The key  estimates for  each of the summands in \eqref{E0} are given in \eqref{induction} and \eqref{E4}.  In the following discussion we will discuss only asymptotic decay as $k\rightarrow \infty$ of the terms in    \eqref{induction} and \eqref{E4}, the dependence on  other variables is the same as in the proof of Theorem  \ref{thm:main}.

If, in fact,  $r+d-\vert \bell^*\vert \geq 2$ then, via Remark \ref{rem:highernegative} (i),  we obtain
$$ \Vert F^{\omega_{d+1}} - F^{\omega_{d+1}, r + d - \vert \bell^*\vert} \Vert_{\cW^{p,d}}\ \leq\
C_{p,d,\ba} \, k^{-2} \,  2^{-(r+d - \vert \bell^*\vert -1)(p-1)} \Vert f \Vert_{\cW^{p+3,d+1}}, $$
i.e., one additional negative power of $k$ compared with \eqref{E4}.
In this case, the corresponding summand  in \eqref{E0} can be estimated as a product of   $d$ terms of $\mathcal{O}(k^{-1})$ (analogous to \eqref{induction}) and one  of $\mathcal{O}(k^{-2})$, yielding
$\mathcal{O}(k^{-(d+2)})$ overall for that summand.

On the other hand, if $r+d-\vert \bell^*\vert =1$,  then, by \eqref{assume},  $\vert \bell^* \vert = r+d-1 \geq 2d+1$. Since also   $\bell^* \geq \mathbf{1}$, it follows that at least one $\ell_j^*$ must be  $\geq 3$. For this $j$,  Remark \ref{rem:highernegative} (ii) can be applied. Thus \eqref{induction} has an estimate of  $\mathcal{O}(k^{-d-1})$ (or better) and combining this with the standard $\mathcal{O}(k^{-1})$ estimate for \eqref{E4}, we again obtain $\mathcal{O}(k^{-d-2})$ overall for that summand as well.

Thus overall  \eqref{E0} has an  $\mathcal{O}(k^{-(d+2)})$ estimate in the case of dimension $d+1$.
This completes the induction argument.
\end{proof}

\begin{remark} \label{relative}
The result in Theorem  \ref{thm:main_expk} is important if one is interested in computing \eqref{highDoscIntegralEq} to high \blue{normalized accuracy (in the sense of \eqref{norm_err})} for large $k$. This is because in general the oscillatory integral \eqref{highDoscIntegralEq} decays with order  $\mathcal{O}(k^{-d})$  as  $k$ increases, and Theorem \ref{thm:main_expk} shows that the \blue{normalized error} in computing this with the FCCS rule still decays with increasing $k$.
\end{remark}

We are now ready to complete the proofs of the  main theorems  - Theorem \ref{thm:main_grid} and \ref{thm:main_grid1},   stated in the Introduction.

\bigskip

\noindent
\emph{Proof of Theorems \ref{thm:main_grid}, \ref{thm:main_grid1}}.  The proofs just require estimating the  term $\displaystyle{\binom{r + d - 1}{d - 1} 2^{-(r-d)(p-1)}}$ which appears in the estimates in  Theorems \ref{thm:main} and \ref{thm:main_expk}.

The reference \cite[Lemma 4]{NoRi:99} gives an  asymptotic  formula (due to M\"{u}ller-Gronbach) for the number of nodes $N(r,d)$ used in the quadrature rule \eqref{FCCSmolyak_intro}. Using this, and a little manipulation, we obtain
$$ N(r,d) \ \approx \ \frac{1}{(d-1)! \, 2^d } \left(1 + \frac{d-1}{r}\right)^{d-1} \left[ 2^r\,  r^{d-1} \right],  $$
where $\approx$ means that the ratio of the  left-hand side and the right-hand side tends to $1$ as $r \rightarrow \infty$. Hence from this it follows that,  for sufficiently large $r$,
\begin{align}  \frac{1}{  (d-1)!\, 2^{d-1}} \left[ 2^r\,  r^{d-1}\right] \geq N(r,d) \ \geq\  \frac{1}{  (d-1)!\, 2^{d+1}} \left[ 2^r\,  r^{d-1}\right]. \label{I2}
\end{align}
In particular, for sufficiently large $r$, we have the inequalities
\begin{align}\label{I3}
  \mathrm{(i)} \quad 2^{-(r-d)} \ \leq \ \frac{2}{(d-1)! } \, \frac{r^{d-1}}{N(r,d)} \quad \text{and}\quad  \mathrm{(ii)} \quad 2^r \ \leq\  N(r,d).
\end{align}
Moreover, as is easily shown,
\begin{align*} \lim \limits_{r \rightarrow \infty} \frac{1}{r^{d - 1}} {\binom{r + d - 1}{d - 1}} =
  \frac{1}{(d - 1)!},
\end{align*}
and so, for sufficiently large $r$,  \begin{align} \label{I4} \binom{r + d - 1}{d - 1} \leq \frac{2}
                                       {(d-1)!} r^{d-1}.  \end{align}
                                     Putting together \eqref{I3} (i)  and \eqref{I4}  and then using \eqref{I3}(ii),  we obtain
                                     \begin{align*} \binom{r + d - 1}{d - 1} 2^{-(r-d)(p-1)} & \leq \left(\frac{2}{(d-1)!}\right)^p \,  (r^{d-1})^p \,  \left( \frac{1}{N(r,d)}\right)^{p-1} \\
   &    \leq \left(\frac{2}{(d-1)!\log^{d-1} 2}\right)^p  \, (\log^{d-1} N(r,d))^p \frac{1}{N(r,d)^{p-1}}  
                                     \end{align*}
                                     Combining this with Theorems \ref{thm:main},  \ref{thm:main_expk}, we complete the proof of Theorems \ref{thm:main_grid}, \ref{thm:main_grid1}.

\begin{remark}[\blue{Higher rate of decay as $k \rightarrow \infty$}] \label{higherorderk}
We remark that the better decay rate as $k \rightarrow \infty$ obtained in Theorem \ref{thm:main_grid1} could be obtained for smaller $r$ if, instead of the mid-point rule when $\ell = 1$ (see  \eqref{eq:grids}), we employ the two-point Clenshaw-Curtis rule. This fact is illustrated in Example 3 of \S \ref{sec:Numerical}.  While this observation is  useful for problems of moderate dimension $d$, it is less interesting for higher $d$ because the number of nodes used in the Smolyak interpolant grows quickly with $d$ if a  rule with more than one point is  employed at level $1$ (see the statement after equation(4) in \cite{BaNoRi:00},
  or  \cite[equation (25)]{NoRi:99}).  
\end{remark}

\begin{remark}[Relation to classical estimates for sparse grid integration]\label{rem:classical}
\blue{  As mentioned in \S \ref{sec:intro},  the result in Theorem \ref{thm:main_grid} can be improved with respect to its dependence on $N(r,d)$ (see \eqref{mainalt2}). The proof of \eqref{mainalt2} can be obtained by following the proof above of Theorem \ref{thm:main_grid}, but using the  following alternative estimate to those given in   Theorem \ref{prop:Errorestimate1dFCC}:}
\begin{align}\label{alter} \blue{\vert \cI^\omega g - \cI^{\omega, \ell} g \vert \ \leq \ C_p \, \left( \frac{1}{n_\ell}\right)^p \,  \Vert g \Vert_{\cW^{p,1}}, \quad \text{for} \quad \ell \geq 1 \quad  \text{and}  \quad \omega \in \mathbb{R} .} \end{align}
\blue{(This obtains a better power of $1/n_\ell$ at the expense of no decay with respect to $\vert \omega\vert $.) Then repeating the proof of Theorem \ref{thm:main_grid} using this estimate instead of Theorem \ref{prop:Errorestimate1dFCC}
yields a  proof of \eqref{mainalt2}. }
\end{remark}

\section{Application to a UQ problem for  the Helmholtz equation}
\label{sec:Apply_Helmholtz}
\noindent
In uncertainty quantification for problems governed by PDEs,  one typically wants to compute the output statistics (e.g. expectation or higher order moments) of  some  quantity of interest (QoI - typically a functional of the   PDE solution), given the statistical properties of the random  input data (e.g., coefficients) of the PDE.    The required output moments are usually   written as a multi-dimensional integral, with dimension determined by the number of random parameters in the model.
  For problems governed by the Helmholtz equation, the solution is usually oscillatory with respect to both the physical variable(s) and  the random  parameters.  
  -- see, e.g., \cite[Chapter 5]{Pe:20}, \cite[\S 4]{GaKuSl:21}, \cite{SpWu:22}.

In this paper  we restrict to problem   \eqref{Helmholtz_Eq_2} -- \eqref{Helmholtz_BC_2} and we
first deal with the oscillation with respect to  $x$ by
applying a   `hybrid numerical-asymptotic' method in physical space.
For each random parameter $\by$,  this yields an expression for the solution  $u(x, \by)$  (increasingly accurate   as $k$ increases) which identifies the principal oscillations with respect to $x$,  and  the only  parts to be computed numerically are smooth functions of $x$. For this Helmholtz problem, it turns out that the approximation also yields the principal
oscillations with respect to $\by$ and  integration with respect to $\by$  requires   integration of multidimensional oscillatory integrals of the type discussed earlier in this paper.   (See e.g., \cite{GrahamSpence:2012}, for a general discussion of hybrid numerical-asymptotic methods in the context of Helmholtz problems.)


Later we will study the case when $n$ is a random field, but we start here
with the deterministic case in order to explain the hybrid numerical-asymptotic method.

\subsection{The deterministic problem and its hybrid numerical-asymptotic solution}
\label{subsec: asymptotic expansion}

In this subsection we shall explain how to obtain  an  asymptotic approximation   for  the solution $u$ of \eqref{Helmholtz_Eq_2} -- \eqref{Helmholtz_BC_2} which is increasingly accurate as $k$ increases. Its implementation will involve the numerical solution of problems which are well-behaved with respect to $k$.

To motivate the idea we first consider problem \eqref{Helmholtz_Eq_2} -- \eqref{Helmholtz_BC_2} in the special case $F=0$. Then, in the \blue{constant-coefficient} case $n=1$ the solution is just a linear combination of the complementary functions $\exp(\pm \ri k x)$. When $n$ is variable,  we define
$N(x) = \int_0^x n(x') \rd x'$ and, for functions $\mu, \nu$  we  consider the `approximate complementary function':
$$r := \mu \xi + \nu \xi^{-1}, \quad \text{where} \quad \xi(x) =  \exp(\ri k N(x)),$$
with $\mu, \nu$ (as yet) unknown functions of $x$.
With this, we readily find:
\begin{align} \label{Lr} L r = \left[\mu'' + \ri k (2n \mu' + n' \mu) \right]\xi  \ + \ \left[ \nu'' - \ri k (2n \nu' + n' \nu) \right]\xi^{-1}.  \end{align}
In order to derive to find an  approximation to $u$ which is accurate for large $k$,  we introduce  the `approximate ray-expansion':
\begin{align} \label{approx_ray} \wu^{m} : = \sum_{j = {0}}^{2m} k^{-j} r_j, \quad \text{where} \quad r_j := \mu_j \xi + \nu_j \xi^{-1}, \quad \text{for} \quad m \geq {0}.  \end{align}  Then (recalling that we assumed $Lu = F = 0$), an easy calculation, using \eqref{Lr}  shows
\begin{align} L(u - \wu^m) &\ = \ {-} \ \sum_{j=0}^{2m} k^{-j}  \left[  \mu_j'' + \ri k (2n \mu_j' + n' \mu_j)
  \right]\xi 
                  \ {-}  \ \sum_{j=0}^{2m} k^{-j} \left[\nu_j'' - \ri k (2n \nu_j' + n' \nu_j) \right]\xi^{-1}. \label{Le} \end{align}
                To force $\wu^m$ to be close to $u$,  we  choose the coefficients $\mu_j$ and $\nu_j$
                so that  the right-hand side of \eqref{Le}  decays rapidly with increasing $k$. Forcing all terms
                 in \eqref{Le} up to  $\mathcal{O}(k^{-2m})$  to vanish leads to the system  of differential equations to be satisfied by $\mu_j, \nu_j$:
  \begin{align}2n \mu_j' + n' \mu_j \ &= \ \ri {\mu_{j-1}^{\prime\prime}}, \quad j = {0}, 1, \ldots ,m  \label{muj} \\2n \nu_j' + n' \nu_j \ &= \ -\ri {\nu_{j-1}^{\prime\prime}},  \quad j = {0}, 1, \ldots ,m \label{nuj}\\
    \text{with} \quad & \mu_{-1} = \nu_{-1} = 0.\label{start}    \end{align}
Then, provided we
choose the boundary conditions for $\mu_j $ and $\nu_j$ so that $u - \wu^m$ satisfies homogeneous boundary conditions and provided the regularity of $\mu_j, \nu_j$ can be controlled we can ensure (at least formally)
that $u - \wu^m  = \mathcal{O}(k^{-2m})$ for any choice of $m$.

This explains the idea in the  simple case $F=0$. Now we prove a  theorem which describes  the general case. To handle $ F \not =0$ we need to introduce the sequence
\begin{align}\label{fj} F_2 = F/n^2 \quad \text{and }  \quad F_{2j+2} = - F_{2j}'' /n^2, \quad \text{for} \quad j = 1,2, \ldots.
 \end{align}

\begin{theorem}
\label{thm: error estimate for asymptotic expanstion}
 Assume that $m \ge 1$, $F \in C^{2m}[0,1]$ and $n \in C^{2m + 2}[0,1]$. Then for $k$ sufficiently large, and for $j = 0, \ldots, 2m$, there exist unique  $\mu_j, \nu_j$ satisfying  \eqref{muj} -- \eqref{start},    together with
  the boundary conditions:
  \begin{align}
    B_L r_0 = u_L, \ & \ B_R r_0 =0 ;\label{BC0}\end{align}
  and,
  \begin{equation}\
\left.    \begin{array}{ll}
B_L r_{2j-1} = 0, \   &\  B_Rr_{2j-1}=0  \\
      B_L r_{2j} = -F_{2j} (0),\  &\ B_R r_{2j} = -B_RF_{2j}
                                    \end{array}
    \right\}, \quad \text{for} \quad  j =  1, \ldots , m.  \label{BC} \end{equation}
  Moreover, with the approximation to $u$ then defined by
  \begin{align} \label{utm} \wu^m \ := \ \sum_{j=0}^{2m} k^{-j} r_j \ + \ \sum_{j=1}^m k^{-2j} F_{2j},  \end{align}    there exists a constant $C$ {independent of $k$} such that, for {sufficiently large} $k$,
  \begin{align} \Vert u - \wu^m \Vert_{{H^1(0,1)}} \ \leq\ C k^{-2m}. \label{asym}
  \end{align}
\end{theorem}
\begin{proof} First,  we note that,  using \cite[Lemma 2.2]{aziz1988two} {and the regularity condition on $n$} we can show by a simple induction that, for $k$ sufficiently large,  $\mu_j$ and $\nu_j$ are well-defined for all $j = 0, \ldots, 2m$ and also $\Vert \mu_{2m}''\Vert_{\infty, [0,1]}$ and  $\Vert \nu_{2m}''\Vert_{\infty, [0,1]} $ are both bounded independently
  of $k$.   Moreover {the regularity conditions of $n$ and $F$ also ensure that} $F_{2j}$ are well defined, {for $j = 1, \ldots, m + 1$},   and $\Vert F_{2m+2}\Vert_{\infty,[0,1]}$ is bounded independently of  $k$.

  Now, to prove the result we observe that by \eqref{Lr}, and then \eqref{muj} --  \eqref{start},
  \begin{align}
    L\left( \sum_{j=0}^{2m} k^{-j} r_j\right) &=  \sum_{j=0}^{2m}\left( k^{-j} \mu_j'' + \ri k^{-(j-1)} (2n \mu_j' + n' \mu_j)\right)\xi \noindent \\& \quad \quad +  \sum_{j=0}^{2m}\left( k^{-j} \nu_j'' - \ri k^{-(j-1)} (2n \nu_j' + n' \nu_j)\right)\xi^{-1}\nonumber\\
                                              &  = \sum_{j=0}^{2m}\left[\left( k^{-j} \mu_j'' - k^{-(j-1)} \mu_{j-1}''\right) + \left( k^{-j} \nu_j'' - k^{-(j-1)} \nu_{j-1}''\right)\right]\nonumber\\
    & = k^{-2m} \left( \mu_{2m}'' + \nu_{2m}''\right). \label{L1}
  \end{align}
  Moreover, using also \eqref{fj},
  \begin{align}     L\left( \sum_{j=1}^{m} k^{-2j} F_{2j} \right) &=
                                                                              \sum_{j=1}^{m} k^{-2j} (F_{2j}'' + k^2 n^2 F_{2j} ) = \ n^2 \sum_{j={1}}^{{m}}  (  k^{-2j + 2} F_{2j} - k^{-2j}  F_{2j+2} )  \nonumber \\
           & = n^2(F_2 - k^{-2m} F_{2m+2} ) = F - k^{-2m} n^2 F_{2m+2}.\label{L2}
       \end{align}
         Combining \eqref{utm}, \eqref{L1} and \eqref{L2},  and using $Lu = F$,   we have
         \begin{align} \label{L3} L(u - \wu^m) = k^{-2m} (n^2F_{2m+2} - \mu_{2m}'' - \nu_{2m}''). \end{align}
         Also, by \eqref{BC0}, \eqref{BC},  we have
         \begin{align} \label{L4} B_L \wu^m & = \sum_{j=0}^{2m} k^{-j} B_L r_j + \sum_{j=1}^m k^{-2j} B_L F_{2j} 
                                                                                                                                = u_L + \sum_{j=1}^m k^{-2j} (-F_{2j}(0) + F_{2j}(0)) = u_L. \end{align}
       and
       \begin{align} \label{L5} B_R \wu^m & = \sum_{j=0}^{2m} k^{-j} B_R r_j + \sum_{j=1}^m k^{-2j} B_R F_{2j} 
                                                                                                                              =  \sum_{j=1}^m k^{-2j} (-B_RF_{2j}  + B_R F_{2j}) = 0.
       \end{align}
       Therefore \begin{align} \label{HomBC}
                   B_L(u - \wu^m ) = 0 = B_R(u - \wu^m), \end{align}
                 and it follows (see e.g., \cite[Lemma 2.1]{aziz1988two}),  that
                 $${\Vert u - \wu^m \Vert_{H^1(0,1)} \leq C k^{-2m}} ,$$
which leads to the required result \eqref{asym}.
\end{proof}

  \begin{remark}\label{rem:app}
    By taking $m = 1$ in Theorem \ref{thm: error estimate for asymptotic expanstion},
     and under the assumptions $n \in C^4[0,1]$ and $F \in C^2[0,1]$, we obtain the approximation
     (valid for $k$ sufficiently large):
     \begin{align} \label{app1}
\wu^1= \widetilde{\mu} \xi + \widetilde{\nu} \xi^{-1} + \widetilde{F},
\end{align}
where    \begin{align} \label{app2} \xi = \exp(\ri k N), \quad \widetilde{\mu} = \mu_0 + k^{-1} \mu_1 + k^{-2} \mu_2,\quad  \widetilde{\nu} = \nu_0 + k^{-1} \nu_1 + k^{-2} \nu_2, \quad \text{and} \quad \widetilde{F} = k^{-2} F_2 = k^{-2}F/n^2. \end{align}
With this approximation we have the error estimate  $|| u - \wu^1 ||_{H^1(0,1)} = \mathcal{O}(k^{-2})$. We use this approximation in the numerical experiment in the following section.
\end{remark}

\subsection{Computation of $\wu^1$}
\label{subsec: computation of mu and nu}
The computation of $\widetilde{F}$ is easy, and thus we need to compute $\widetilde{\mu}$ and $\widetilde{\nu}$ in order to obtain $\wu^1$, which is then reduced to the computation of $\{ \mu_j, \nu_j \}_{j = 0}^2$. We know from \secref{subsec: asymptotic expansion} that $\mu_j$ and $\nu_j$ satisfy the ODE system \eqref{muj} -- \eqref{start} together with the boundary conditions \eqref{BC0} and \eqref{BC} for $j = 0, 1, 2$. The ODE system \eqref{muj} -- \eqref{start} admits the following solutions:
\begin{align}
  \label{tot_soln}
\mu_j(x) = \alpha_j^1 \mu_j^G(x) + \mu_j^P(x), \quad \nu_j(x) = \alpha_j^2 \nu_j^G(x) + \nu_j^P(x),
\end{align}
where \begin{align} \label{gen_soln} \mu_j^G(x) = \nu_j^G(x) = n(x)^{-\frac{1}{2}}\end{align}
are the general solutions and
\begin{align}
\mu_j^P(x) = \frac{i}{2} n(x)^{-\frac{1}{2}} \int_0^x \mu_{j - 1}^{\prime\prime}(s) n(s)^{-\frac{1}{2}} ds, \quad \nu_j^P(x) = - \frac{i}{2} n(x)^{-\frac{1}{2}} \int_0^x \nu_{j - 1}^{\prime\prime}(s) n(s)^{-\frac{1}{2}} ds
\end{align}
are the particular solutions, with the coefficients $\alpha_j^1$ and $\alpha_j^2$ determined by the boundary conditions \eqref{BC0} and \eqref{BC}. We let
\begin{align} \label{eq: integrals in mu and nu}
\cI_j^{\mu}(x) = \int_0^x \mu_{j - 1}^{\prime\prime}(s) n(s)^{-\frac{1}{2}} ds, \quad \cI_j^{\nu}(x) = \int_0^x \nu_{j - 1}^{\prime\prime}(s) n(s)^{-\frac{1}{2}} ds.
\end{align}
In general, we do not have the analytic expressions for $\mu_j$ and $\nu_j$, since we do not always have an explicit expression for the integrals in \eqref{eq: integrals in mu and nu}. We sometimes need to resort to numerical integration.

\paragraph{Partition of $[0, 1]$} Let $M$ and $L$ be positive integers with $H = \frac{1}{M}$ and $h = \frac{1}{LM}$, where $L$ is assumed to be even \blue{for convenience}. We partition $[0, 1]$ into $0 = x_0 < x_1 < \ldots < x_{M - 1} < x_M = 1$, where $x_m = mH$ for $m = 0, 1, \ldots, M$. For each interval $[x_m, x_{m + 1}]$, we further partition it into $x_m = x_m^0 < x_m^1 < \ldots < x_m^{L - 1} < x_m^L = x_{m + 1}$, where $x_m^{\ell} = x_m + \ell h$ for $\ell = 0, 1, \ldots, L$. We aim to obtain the values of $\widetilde{\mu}$ and $\widetilde{\nu}$ at $x_m$ for $m = 0, 1, \ldots, M$.

\paragraph{Computation of $\mu_0$ and $\nu_0$} Note that $\mu_0^P = \nu_0^P = 0$ since $\mu_{-1} = \nu_{-1} = 0$. Then $\mu_0(x)$ and $\nu_0(x)$ can be analytically obtained with the coefficients $\alpha_0^1$ and $\alpha_0^2$ determined by the boundary condition \eqref{BC0}.

\paragraph{Computation of $\mu_1$ and $\nu_1$} Note that $\mu_0^{\prime\prime}(x) n(x)^{-\frac{1}{2}}$ and $\nu_0^{\prime\prime}(x) n(x)^{-\frac{1}{2}}$, i.e., the integrands of $\cI_1^{\mu}$ and $\cI_1^{\nu}$, are analytically available. We obtain the values of $\cI_1^{\mu}(x_m^{\ell})$ and $\cI_1^{\nu}(x_m^{\ell})$ for each $m$ and $\ell$ by successively applying an $M_G$-point Gauss quadrature rule to the integration on $[x_m^{\ell}, x_m^{\ell + 1}]$. Then we obtain the values of $\mu_1(x_m^{\ell})$ and $\nu_1(x_m^{\ell})$ for each $m = 0, 1, \ldots, M$ and $\ell = 0, 1, \ldots, L$, with the coefficients $\alpha_1^1$ and $\alpha_1^2$ determined by the boundary condition \eqref{BC} for $j = 1$.

\paragraph{Computation of $\mu_2$ and $\nu_2$} Note that
\begin{align} \label{eq: computation of mu1P}
\left( \mu_1^P(x) \right)^{\prime\prime} = \frac{i}{2} \left[ \left( \frac{3}{4} n(x)^{-\frac{5}{2}} n^{\prime}(x)^2 - \frac{1}{2} n(x)^{-\frac{3}{2}} n^{\prime\prime}(x) \right) \cI_1^{\mu}(x) - \frac{3}{2} n(x)^{-2} n^{\prime}(x) \mu_0^{\prime\prime}(x) + n(x)^{-1} \mu_0^{\prime\prime\prime}(x) \right].
\end{align}
Hence $\mu_1^{\prime\prime}(x) n(x)^{-\frac{1}{2}}$ is available at $x_m^{\ell}$ for each $m = 0, 1, \ldots, M$ and $\ell = 0, 1, \ldots, L$, and so is $\nu_1^{\prime\prime}(x) n(x)^{-\frac{1}{2}}$. We obtain the values of $\cI_2^{\mu}(x_m)$ and $\cI_2^{\nu}(x_m)$ for each $m$ by successively applying the composite Simpson's rule to the integration on $[x_m, x_{m + 1}]$ using the values of $\mu_1^{\prime\prime}(x) n(x)^{-\frac{1}{2}}$ and $\nu_1^{\prime\prime}(x) n(x)^{-\frac{1}{2}}$ at $x_m^{\ell}$ for $\ell = 0, 1, \ldots, L$. The we obtain the values of $\mu_2(x_m)$ and $\nu_2(x_m)$ for each $m = 0, 1, \ldots, M$, with the coefficients $\alpha_2^1$ and $\alpha_2^2$ determined by the boundary condition \eqref{BC} for $j = 2$.

\subsection{The random problem}
  \label{subsec:random}
  We now introduce the random model by assuming that $n(x)$ depends in an affine way on
  $d$ i.i.d. random variables $\by=(y_1,...,y_d)$ 
with $y_i\in U[-1,1]$.
That is, we assume
\begin{align} \label{eq: parametrization of refractive index}
n(x,\by)=n_0(x) + \sum_{j=1}^{d}n_j(x)y_j,
\end{align}
where $ n_j(x) \in C^{4}{[-1, 1]} $ for $j = 0, 1, \ldots, d$. We also assume that the expansion functions satisfy, for some constant $C>0$,

  \begin{align} \label{pos} \min_{x \in [0,1]} n_0(x)  - \sum_{j=1}^d ||n_j||_{\infty, [0,1]} \ \ge\  C
  \end{align}
  This condition ensures that
  \begin{align}  C \ \leq\  n(x,\by) \ \leq \  \sum_{j=0}^d \Vert n_j \Vert_{\infty, [0,1]},    \quad \text{for all} \quad  x \in [0,1], \quad \by \in U[-1,1]^d. \label{uniform}
  \end{align}


With the parametrization of $n$ given in \eqref{eq: parametrization of refractive index}, we have
\begin{align}
N(x,\by) = N_0(x)  + \sum_{j=1}^{d} N_j(x)y_j, \quad \text{with} \quad   N_j(x) = \int_0^x n_j(x')\rd x', \quad j = 0, \ldots, d,
\end{align}
}
and hence
  \begin{align}
    \xi(x, \by) = \exp(\ri k N_0(x)) \exp(\ri k \ba(x)\cdot \by), 
  \end{align}
  where $ \ba \in C[0,1]^d$ is the real vector-valued function with components given by
\begin{align} a_j(x) = N_j(x), \quad j = 1, \ldots , d. \label{defa} \end{align}

In UQ applications one is often  interested  in computing  the expectation or higher moments of a Quantity of Interest,  typically a  functional of the solution  $u$. Since our aim here is to provide an application of the use of the FCCS rule combined with the hybrid numerical-asymptotic method we restrict attention to the computation of the the expectation of $u(x)$ at any given point $x \in [0,1]$ which we can approximate by
  $\mathbb{E}[\wu^1(x)]$, with ${\wu}^1$ given in \eqref{app1}, \eqref{app2} above. To express this quantity neatly, for any smooth enough function $\xi$ defined on $[0,1]\times[-1,1]^d$,   we define the integrals
  \begin{align} (\cI^{\pm \ba} \xi)(x) = 2^{-d}  \int_{[-1,1]^d} {\xi}(x,\by) \exp(\pm \ri k \ba(x) \by)d\by \quad \text{and} \quad (\cI \xi)(x) = 2^{-d}  \int_{[-1,1]^d} {\xi}(x,\by) d\by.
  \end{align}
  Using this notation we have
  \begin{align}
    \mathbb{E}[\wu^1(x)] = & \  \ \ \exp(\ri k N_0(x)) \, (\cI^{+\ba}\widetilde{\mu})(x) \label{intmu} \\
    & +  \exp(- \ri k N_0(x)) \,  (\cI^{-\ba}\widetilde{\nu})(x) \label{intnu} \\ 
          &  +  \, (\cI \widetilde{F}) (x)  \label{intn} 
\end{align}
The integral \eqref{intn} is not oscillatory and can be computed using the standard Clenshaw-Curtis-Smolyak rule,  
but  \eqref{intmu} and \eqref{intnu} need to be computed   using our FCCS rule.  The values of  $\widetilde{\mu}(x, \by)  $ and $\widetilde{\nu}(x,\by)$ at the multidimensional
quadrature nodes are computed using the procedure described in \S \ref{subsec: computation of mu and nu}.

At this point we remark that, although the functions $\mu_j, \nu_j$ are non-oscillatory with respect to spatial variable $x$, the coefficients $\alpha_j^\ell$ appearing in \eqref{tot_soln} do depend on $k$ (via application of the boundary conditions \eqref{BC0}, \eqref{BC}), and so $\widetilde{\mu}, \widetilde{\nu}$ could potentially have
  derivatives (with respect to the random \blue{parameters} $\by$)  which depend on $k$.
We investigate this point in \secref{subsec: numerical examples of UQ problem} and show that,  under reasonable assumptions, the application of the FCCS rule to \eqref{intmu}, \eqref{intnu},  where  $\widetilde{\mu}, \widetilde{\nu}$  are approximated  by the Smolyak \blue{interpolation}  can be justified and works well in practice.

\blue{\begin{remark}[Possible use of asymptotic expansions for integrals] When all components of   $k\ba$ are large enough, one could consider using
  the known asymptotic expansions of oscillatory integrals with linear phase (found, e.g., in \cite[\S 2.4]{DeHuIs:18}) to approximate \eqref{intmu},
  \eqref{intnu}, as an alternative to   the FCCS rule. However  in our  UQ application,  the vector $\ba(x)$ can easily have components which are small or even vanish,  and the location of these  ``non-oscillatory'' directions can vary with $x$ in a complicated way.   Because of this, even when $k$ is large, some of the  products $ka_j(x)$ may still be quite small (or even vanish) and   the corresponding integrals \eqref{intmu},
  \eqref{intnu}  will not be (highly) oscillatory in the corresponding dimensions. Moreover the location of such components also depends on the choice of expansion functions $n_j$. Thus in order to have a method which works robustly in $k$ and $\ba$  we prefer to use the FCCS rule, and we have designed this so that it works robustly independently of the size of $ka_j(x)$ (see
 \eqref{highDoscIntegralEq1}).  Moreover, to obtain high accuracy in the asymptotic evaluation of oscillatory integrals, it is necessary to compute partial derivatives of the smooth function $f$ (in (1.1)), a task which would be complicated in the context of our UQ application.\end{remark}}

\subsection{\blue{More general Quantities of Interest}}

\blue{In the computations in the following section we compute the expected value of the
  linear functional $u(1)$. This has a certain physical meaning, since the boundary condition \eqref{Helmholtz_BC_2} at the point $x=1$ corresponds to an approximation of the Sommerfeld radiation condition at this point. Thus the value $u(1)$ constitutes
  an approximation to the value of the wavefield in the `far field', a quantity of some physical interest.}

 The approximate computation of $\mathbb{E}(u)$  can easily be extended to general linear functionals on $H^1(0,1)$. For example,  to approximate the expected value of $Gu : = (u,g)_{L^2(0,1)}$, with $g \in L^2(0,1)$, we compute
\begin{align} & \mathbb{E}\left[ \int_0^1 \widetilde{u}^1(x, \cdot) \overline{g}(x) \rd x\right] \  = \  \int_0^1 \mathbb{E}\left[\widetilde{u}^1(x, \cdot)\right] \overline{g}(x) \rd x \nonumber \\ &\mbox{\hspace{1cm}} = \ \int_0^1 \exp(\ri k N_0(x))  (\cI^{+\ba} \widetilde{\mu}) (x)  \overline{g}(x) \rd x  + \int_0^1 \exp(- \ri k N_0(x))  (\cI^{-\ba} \widetilde{\nu}) (x)  \overline{g}(x) \rd x + \int_0^1  (\cI \widetilde{f})(x)  \overline{g}(x) \rd x ,
                              \label{addosc}
 \end{align}
so that  additional 1d oscillatory integrals now appear in the first and second terms of \eqref{addosc}. Similar but slightly more complicated terms appear in the computation of the general linear functional on $H^1(0,1)$:
                            $Gu : = (u',g')_{L^2(0,1)} + (u,g)_{L^2(0,1)}$, for any $ g \in H^1(0,1)$. }

                            \blue{Some nonlinear functionals  can also  be handled by the above technique. For example, to compute the variance of $u(x)$, we need also $\mathbb{E}[u(x)^2]$. We approximate this by the expectation of 
                              $$ (\widetilde{u}^1(x))^2 = \widetilde{\mu}^2 \xi^2 + \widetilde{\nu}^2 \xi^{-2} + 2 (\widetilde{\mu} \widetilde{F}) \xi + 2 (\widetilde{\nu} \widetilde{F}) \xi^{-1} + \left(\widetilde{F}^2 + 2 \widetilde{\mu} \widetilde{\nu}\right), $$
                              and the integral with respect to $\by$ of all these terms can be done by the FCCS rule. (The last two terms are not oscillatory.)
                              Similar arguments can be used to approximate the expected value of $\vert u(x) \vert^2 = u(x) \overline{u(x)}$.  }

\section{Numerical Experiments}\label{sec:Numerical}

Our code for the following numerical examples was based on the Sparse Grids Matlab Kit, available at {\tt https://sites.google.com/view/sparse-grids-kit}. See also  \cite{back2011stochastic,piazzola2022sparse}.

\subsection{Multi-dimensional quadrature}
\label{subsec:quad} In this subsection we illustrate the convergence properties of the Filon-Clenshaw-Curtis-Smolyak rule.  \blue{We shall compute the following error indicators},  
  defined as:
  $$ e^{k,d, \ba,r}(f) = \left\vert \cI^{k,d,\ba}f -  \cI^{k,d,\ba,r}f\right \vert \quad \text{and} \quad E^{k,d, \ba,r}(f) = \frac{\left\vert \cI^{k,d,\ba}f -  \cI^{k,d,\ba,r}f\right \vert}{\left\vert \cI^{k,d,\ba}f\right\vert}, $$
with $\cI^{k,d, \ba}f$ and $\cI^{k,d, \ba,r }f$ defined in \eqref{highDoscIntegralEq1} and \eqref{FCCSmolyak_highDoscIntegralEq2} respectively. \blue{Here $e^{k,d, \ba,r}(f)$ is the absolute error and $E^{k,d, \ba,r}(f)$ is a proxy for the normalized error (as defined in \eqref{norm_err}). We use the proxy $E^{k,d, \ba,r}(f)$ because in some experiments the number of oscillatory dimensions $\widetilde{d}$ changes within the experiment and the quantity $k^{- \widetilde{d}}$ is well represented by $| \cI^{k,d,\ba}f |$ according to the discussion above \eqref{norm_err}.}

\paragraph{Example 1: Exactness check.}
Recall that by \cite[Proposition 2]{BaNoRi:00},  we have  $\cQ^{r,d} f = f$,  for all
\begin{align}\label{exactness}    f \in \sum_{\vert \bell \vert = r + d -1} \bP(n_{\ell_1}) \otimes  \bP(n_{\ell_2}) \otimes \ldots \otimes  \bP(n_{\ell_d}),
\end{align}
where $n_1 = 0$  and $n_\ell = 2^{\ell-1}$ for $\ell \geq 2$ and  $\bP(n) $ denotes the univariate polynomials of degree $n$.     Choosing
    \begin{align} \label{choosef}
    f(\by) = \prod_{j=1}^d y_j^2,
\end{align}
 we have $\cQ^{r,d} f = f $ when  $r \geq d+1$. (This is   because    $\bell := (2,2,\ldots, 2)^\top$  satisfies $\vert \bell \vert = 2d \leq r+d-1$.)
Results for $d = 4$, $\ba = (1,0,1,0)^\top$ are given in Table \ref{tab:0} and clearly indicate exactness of the approximation for $r \geq 5 = d+1$. 
  \begin{table}[H]
\begin{center}
\begin{tabular}{|l|l|l|l|l|l|l|l|}
\hline
\multicolumn{8}{|c|}{$e^{k,d,\ba,r} (f)$ } \\
\hline
$r$ & 1 & 2 & 3 & 4 & 5 & 6 & 7 \\
\hline
$ k = \pi/2$ & 2.59 (-2) & 2.59 (-2) & 2.59 (-2) & 2.59 (-2) & {3.47 (-17)} & {2.09 (-17)} & {3.71 (-17)} \\
$ k = 2\pi$ & 4.56 (-3) & 4.56 (-3) & 4.56 (-3) & 4.56 (-3) & {6.07 (-18)} & {1.65 (-17)} & {2.33 (-17)} \\
\hline
\end{tabular}
\caption{\label{tab:0} $e^{k,d,\ba,r} (f)$ for $d = 4$, $k = \pi/2, 2\pi$ with $f$ as in \eqref{choosef}}
\end{center}
\end{table}

In the following three  examples we compute \eqref{highDoscIntegralEq} for the case $d = 3$ with \begin{align} \label{deff} f(\by) = \cos(m y_1y_2y_3)\end{align} and various choices of $\ba$.   The exact value of $\cI^{k,d,\ba}f$ is taken to be
  $\cI^{k,d,\ba,10}f$, and this value is used to compute \blue{the error indicators}. We repeated with exact value computed with   $r=12$  and observed no changes in the results.

\paragraph{Example 2 - fast convergence with increasing $r$.} \label{expt1}

This experiment illustrates the fast decay of the error as   $r$ increases  when $f$ is smooth.
The \blue{normalized error proxies $E^{k,d, \ba,r}(f)$} for four  different values of $m$ as $r$ increases and $k = 101.53$ is fixed
are given in Table \ref{tab:1}. Convergence as $r$ increases is observed in the columns corresponding to $m = 2,4,8$.
For fixed $r$,  \blue{$E^{k,d, \ba,r}(f)$} grows roughly proportional to  $m$.
In the column headed $m=16$ we see that convergence has hardly started yet, and higher $r$ will be needed to see convergence. Intuition for this can be obtained from Theorem \ref{thm:main_grid1}, since  $\Vert f \Vert_{W^{p,d}} \sim m^p$, indicating the need for higher $r$ before convergence commences.

  \bigskip

  \begin{table}[H]
  \begin{center}
    \begin{tabular}{|l|l|l|l|l|}
      \hline
      & $m = 2$ & $m=4$ & $m=8$ & $m=16$     \\
      \hline
      $r$ & $E^{k,d,\mathbf{a},r}(f)$  & $E^{k,d,\mathbf{a},r}(f)$  & $E^{k,d,\mathbf{a},r}(f)$  & $E^{k,d,\mathbf{a},r}(f)$  \\
      \hline
      3 & 3.22 (+0) & 2.67 (+0) & 4.32 (+0) & 2.10 (+0) \\
      4 & 4.10 (-2) & 1.99 (-1)  & 3.73 (-1) & 1.37 (-1) \\
      5 & 2.20 (-3)  & 7.13 (-2) & 1.90 (-1) & 1.83 (-1) \\
      6 & 9.47 (-5) & 2.25 (-3) & 5.87 (-2) & 1.62 (-1) \\
      \hline
    \end{tabular}
    \caption{\label{tab:1} \blue{$E^{k,d, \ba,r}(f)$} as  $r$ increases when $\mathbf{a} = (1,1,1)^\top$, $k = 101.53$ and
      $f$ given by \eqref{deff}}
  \end{center}
\end{table}

    \bigskip

    \paragraph{Example 3 - asymptotic decay  as $k \rightarrow \infty$, with fixed $r$. } \label{expt2}
Again we fix $\ba = (1,1,1)^\top$ and   study the behaviour
  of \blue{the error indicators} as   $k\rightarrow \infty$. In general,  the decay rate of the exact and approximate integrals  as $k \rightarrow \infty$ can be quite delicate.  To see this, consider the following model 1D integral with integration  by parts:
\begin{align} \label{int_parts} \int_{-1}^1 \exp( \ri k y) g(y) \rd y = \frac{1}{\ri k} \left[g(1) \exp(\ri k) - g(-1) \exp(-\ri k)\right] - \frac{1}{\ri k} \int_{-1}^1
\exp( \ri k y) g'(y) \rd y.  \end{align}
A second integration by parts shows that  the second term on the right-hand side of \eqref{int_parts} is  $\mathcal{O}(k^{-2})$, while   the first term   takes the form  $C(k) k^{-1}$, so is dominant in general.
However in general the factor $C(k)$ can  vary considerably with respect to $k$,
leading to possibly irregular behaviour  as $k\rightarrow \infty$.
However  by taking  $k = 2 \ell \pi + \pi/4$, $\ell = 2,4,8,\ldots, 128$,  $C(k)$ turns out to be the $k-$independent constant $C(k) = ((1-\ri) g(1)+ (1+\ri) g(-1))/\sqrt{2}$, thus
ensuring (excluding special cases of $g$) a regular $O(k^{-1})$  decay for the dominant term in \eqref{int_parts}. We use this sequence of wavenumbers in the experiments below.
With  $f $ as given in \eqref{deff} with $m = 2$,   Table \ref{tab:2}  
illustrates that  $\vert \cI^{k,d,\ba} f \vert $  decays with $\mathcal{O}(k^{-d})$.
    Also, \blue{$E^{k,d, \ba,r}(f)$} remains bounded with respect to $k$ when  $r  = 3$ and decays with order \blue{about}  $\mathcal{O}(k^{-1})$ for $r =  4$, as predicted by Theorem \ref{thm:main_grid1}.

 \begin{table}[H]
   \begin{center}
 \begin{tabular}{|l|ll|ll|ll|ll|ll|}
   \hline
    & & &\multicolumn{4}{c|}{$r=3$}& \multicolumn{4}{c|}{$r=4$}\\
   \hline
   $k$  &$\vert \cI^{k,d,\mathbf{a}}(f)\vert $& \blue{$\alpha$} &        $e^{k,d,\mathbf{a},r}(f)$    &       \blue{$\alpha$}    &  $E^{k,d,\mathbf{a},r}(f)$ & \blue{$\alpha$}   & $e^{k,d,\mathbf{a},r}(f)$   &  \blue{$\alpha$}  &  $E^{k,d,\mathbf{a},r}(f)$   & \blue{$\alpha$}\\
   \hline
13.35  & 1.06(-03) &       & 2.25(-03) &      & 2.12(+00) &      & 2.35(-04) &       & 2.21(-01) &      \\
25.92  & 1.04(-04) & \blue{3.50} & 2.66(-04) & \blue{3.22} & 2.56(+00) & \blue{-0.29} & 1.88(-05) & \blue{3.80} & 1.81(-01) & \blue{0.30} \\
51.05  & 1.12(-05) & \blue{3.29} & 3.24(-05) & \blue{3.11} & 2.90(+00) & \blue{-0.18} & 1.28(-06) & \blue{3.97} & 1.14(-01) & \blue{0.68} \\
101.32 & 1.28(-06) & \blue{3.16} & 4.00(-06) & \blue{3.05} & 3.12(+00) & \blue{-0.11} & 8.22(-08) & \blue{4.00} & 6.42(-02) & \blue{0.84} \\
201.85 & 1.52(-07) & \blue{3.09} & 4.96(-07) & \blue{3.03} & 3.26(+00) & \blue{-0.06} & 5.20(-09) & \blue{4.00} & 3.41(-02) & \blue{0.92} \\
402.91 & 1.86(-08) & \blue{3.05} & 6.18(-08) & \blue{3.01} & 3.33(+00) & \blue{-0.03} & 3.27(-10) & \blue{4.00} & 1.76(-02) & \blue{0.96} \\
805.03 & 2.29(-09) & \blue{3.02} & 7.71(-09) & \blue{3.01} & 3.36(+00) & \blue{-0.02} & 2.05(-11) & \blue{4.00} & 8.94(-03) & \blue{0.98} \\
   \hline
 \end{tabular}
 \caption{ \label{tab:2} \blue{Error indicators} for   $\mathbf{a} = (1,1,1)^\top$, $f$ given by \eqref{deff} with $m=2$, increasing $k$, for   $r=3$ and $r=4$. \blue{The decay rate is conjectured to be $Ck^{-\alpha}$, with values of $\alpha$ computed from successive pairs of computations}  }
 \end{center} \end{table}

{We recall that the rule analysed in this paper uses the mid-point rule  at level 1, and Clenshaw-Curtis grids thereafter (see \eqref{eq:grids}).   By Remark \ref{higherorderk}, 
  a better asymptotic decay of \blue{$E^{k,d, \ba,r}(f)$} with respect to  $k$ --  even for small $r$ --  can be obtained
  if we use instead  the two-point Clenshaw-Curtis  rule at level 1.  Results using this rule for
  the same case as in Table \ref{tab:2} are given in Table \ref{tab:3}.
  Here \blue{$E^{k,d, \ba,r}(f)$} is observed to decay with \blue{$\mathcal{O}(k^{-2})$ when  $r = 4$  although the behaviour for   $r = 3$ has not stabilized yet. Overall we observe  several orders of magnitude improvement in accuracy for large $k$.}
  However we recall  that the number of quadrature points grows more rapidly with dimension when the level 1 rule has more than one point and so this method may not be appropriate for higher dimensions. But if $d$ is not too big,   this variant should be  useful   if accurate results for very high $k$ are required.
 }

 \begin{table}[H]
   \begin{center}
 \begin{tabular}{|l|ll|ll|ll|ll|}
   \hline
     &\multicolumn{4}{c|}{$r=3$}& \multicolumn{4}{c|}{$r=4$}\\
   \hline
   $k$    &        $e^{k,d,\mathbf{a},r}(f)$    & \blue{$\alpha$}          &  $E^{k,d,\mathbf{a},r}(f)$ &  \blue{$\alpha$}  & $e^{k,d,\mathbf{a},r}(f)$   &  \blue{$\alpha$}  &  $E^{k,d,\mathbf{a},r}(f)$   & \blue{$\alpha$} \\
   \hline
  13.35 & 6.65(-05) &       & 6.27(-02) &      & 2.05(-05) &       & 1.93(-02) &      \\
  25.92 & 2.57(-06) & \blue{4.90} & 2.47(-02) & \blue{1.40} & 8.37(-07) & \blue{4.82} & 8.06(-03) & \blue{1.32} \\
  51.05 & 5.36(-08) & \blue{5.71} & 4.79(-03) & \blue{2.24} & 2.86(-08) & \blue{4.98} & 2.56(-03) & \blue{1.69} \\
 101.32 & 1.03(-09) & \blue{5.77} & 8.05(-04) & \blue{2.60} & 9.25(-10) & \blue{5.01} & 7.23(-04) & \blue{1.84} \\
 201.85 & 2.19(-10) & \blue{2.25} & 1.43(-03) & \blue{-0.83} & 2.93(-11) & \blue{5.01} & 1.92(-04) & \blue{1.92} \\
 402.91 & 1.88(-11) & \blue{3.55} & 1.01(-03) & \blue{0.50} & 9.19(-13) & \blue{5.01} & 4.94(-05) & \blue{1.96} \\
 805.03 & 1.34(-12) & \blue{3.82} & 5.83(-04) & \blue{0.79} & 2.85(-14) & \blue{5.02} & 1.24(-05) & \blue{2.00} \\
   \hline
 \end{tabular}
 \caption{ \label{tab:3} \blue{Error indicators} for   $f$ given in \eqref{deff} with $m = 2$, $\mathbf{a} = (1,1,1)^\top$,  for $r=3$ and $r=4$  using Clenshaw-Curtis 2 point rule on level 1. \blue{The decay rate is conjectured to be $Ck^{-\alpha}$, with values of $\alpha$ computed from successive pairs of computations}}
 \end{center} \end{table}


\paragraph{Example 4 - Robustness to variation in the elements of $\ba$.}
Recall that integral $\cI^{k,d,\ba}(f)$  is only oscillatory in the $j$th dimension when $k \vert a_j \vert \geq 1$,
and if this is not true then the standard quadrature rule (and not its  Filon variant) is applied in that dimension.
The standard rule is  Clenshaw-Curtis when $\ell \geq 2$ and the mid-point rule when $\ell = 1$. Here we show that the algorithm proposed works stably when $k \vert a_j \vert$ passes through the value $1$ or when $a_j = 0$. In this example we consider examples  with   $\ba = (0.01,1,1)^\top$ and $\ba = (0,1,1)^\top$. In the first case
$\cI^{k,d,\ba}(f)$ is not oscillatory in the $y_1$ direction unless
$k\geq 100$ and in the second case it is never oscillatory in the $y_1$ direction. Results are in Table \ref{tab:4}.
The method behaves robustly with respect to the value of $a_1$. 



\begin{table}[H]
  \begin{center}
\begin{tabular}{lr}

    \begin{tabular}{|c|c|c|c|}
      \hline
      & \multicolumn{3}{c|}{$\ba = (0.01,1,1)^\top$}\\
      \hline
      & $k = 25.92$ & $k = 101.32$ & $k = 201.85$      \\
      \hline
      $\vert \cI^{k,d,\ba} (f)\vert \ : $ & 2.30(-3) & 1.68(-4) & 3.96(-5)  \\
      \hline
      $E^{k,d,\ba,r}(f)$: &&&\\
      $r=4$ & 1.96(-1) & 1.34(-1) & 5.42(-2)\\
      $r=5$ & 2.41(-2) & 7.00(-3) & 3.54(-3)\\
      $r=6$ & 1.37(-4) & 2.70(-4) & 4.57(-6) \\
      $r=7$ & 1.30(-5) & 2.13(-5) & 1.92(-5)\\
      $ r=8$ & 2.05(-6) & 4.46(-7) & 1.59(-7)\\
      \hline
    \end{tabular}
&
    \begin{tabular}{|c|c|c|}
      \hline
       \multicolumn{3}{|c|}{$\ba = (0,1,1)^\top$}\\
      \hline
       $k = 25.92$ & $k = 101.32$ & $k = 201.85$  \\
      \hline
      2.30(-3)         & 1.70(-4)  & 4.38(-5)   \\
      \hline
      &&\\
1.80(-1)&1.64(-1) &1.63(-1)\\
        2.47(-2) & 7.97(-3) &4.87(-3) \\
         2.11(-4)&3.88(-4)  & 2.21(-4) \\
         1.56(-5) & 1.53(-5) & 1.09(-5)\\
         2.12(-6) & 8.60(-7)  & 2.48(-7)\\
             \hline
    \end{tabular}

\end{tabular}
    \caption{\label{tab:4} Values of \blue{$E^{k,d, \ba,r}(f)$} as $r$ increases, with $d=3, m = 2$ and two different choices of $\ba$. }
\end{center} \end{table}


In the next   example we illustrate the advantage arising when some dimensions of the problem are less important than others.

\paragraph{Example  5 - Decaying importance of dimensions}


{
  Here we consider the functions
\begin{align} 
 f(\by) & = \cos(m y_1 y_2) \cos(m y_3 y_4) \cos( m y_5 y_6), \label{hd1} \\
f(\by) & = \cos(m y_1 y_2) \cos(0.1 m y_3 y_4) \cos(0.01 m y_5 y_6).\label{hd2}
\end{align}
With  $k = 16 \pi + 1 \approx 51.27$ and $\ba = (1, \ldots, 1)^\top \in \mathbb{R}^6$,
the  reference values of the integrals are  taken to be the product of the quadrature approximations to the three 2-d integrals:
$$\cI^{k,6,\ba} (f) \ \approx\  \prod_{j = 1}^3 \cI^{k,2,(a_{2j - 1}, a_{2j}),10} (f_j).$$
\blue{$E^{k,d, \ba,r}(f)$} for each choice of $f$ are given in Table \ref{tab:5}}, illustrating   the substantial benefit of the decay in importance of the dimensions in case \eqref{hd2} compared with case \eqref{hd1}.



  \begin{table}[H]
  \begin{center}
    \begin{tabular}{|l|l|l|l|l|}
      \hline
      & $m = 1$ & $m = 2$ & $m = 3$ & $m=4$     \\
      \hline
      $r$ & $E^{k,d,\mathbf{a},r}(f)$  & $E^{k,d,\mathbf{a},r}(f)$  & $E^{k,d,\mathbf{a},r}(f)$  & $E^{k,d,\mathbf{a},r}(f)$  \\
      \hline
      6 & {7.92 (-1)} & {3.14 (+1)} & {7.78 (+0)} & {1.74 (+1)} \\
      7 & {8.51 (-3)} & {1.49 (+0)} & {1.00 (+0)} & {6.15 (+0)} \\
      8 & {4.47 (-5)} & {8.51 (-2)} & {1.68 (-1)} & {2.52 (+0)} \\
      9 & {3.21 (-6)} & {3.62 (-4)} & {7.35 (-3)} & {3.71 (-1)} \\
      \hline
      \hline
      6 & 3.52 (-7) & 2.27 (-5) & 4.27 (-4) & 6.51 (-3) \\
      7 & 7.84 (-9) & 1.93 (-6) & 1.67 (-5) & 1.56 (-4) \\
      8 & {6.74 (-10)} & 1.35 (-7) & 6.33 (-7) & 6.04 (-6) \\
      9 & {2.61 (-12)} & {8.76 (-10)} & {3.23 (-8)} & 1.18 (-6) \\
      \hline

    \end{tabular}
    \caption{\label{tab:5} \blue{$E^{k,d, \ba,r}(f)$} as  $r$ increases when  $d = 6$, $\mathbf{a} = (1, \ldots, 1)^\top$, $k = 51.27$ and $f$ given by \eqref{hd1} (top panel) and \eqref{hd2} (bottom panel)}
  \end{center}
\end{table}



\paragraph{Example 6 -- Dimension-adaptive methods.}
It is well-known (e.g., \cite{GeGr:03,nobile2016adaptive}) (and the previous example shows) that if the dimensions can be ordered so that higher dimensions become less and less important than lower dimensions, then dimension-adaptive tensor product methods will be more efficient than standard procedures. This observation is relevant to the UQ problem considered in the next subsection. To illustrate this here, we consider the integral
\begin{align} \label{academic}
\cI(x) : = \int_{[-1,1]^d} n^{-1/2} (x, \by) \exp(\ri k \ba(x). \by) \rd y,
\end{align}
where $n$ given in \eqref{eq: parametrization of refractive index}, with
\begin{align} \label{def_n}  n_0(x)  = 1 \quad \text{and} \quad n_j(x) = \exp(-j) \sin (j \pi x), \quad \text{for} \quad x \in [0,1], \end{align}
and so, by \eqref{defa},
$$
a_j(x) \ =\  N_j(x) \ = \ \int_0^{x} n_j(x') \rd x'  \ = \ \exp(-j) \int_0^{x} \sin (j \pi x') \rd x' = \frac{1}{j \pi}\,  \exp(-j)\, (1 - \cos(j \pi x)).
$$
Since the functions $\mu_0, \nu_0$ (which constitute the principal  parts of $\widetilde{\mu}, \widetilde{\nu}$ in \eqref{app2}) are  $\by$-dependent multiples of $n^{-1/2}$, the computation of \eqref{academic} 
is a good test for the UQ computation considered in the following subsection.
In this example we choose $k = 101.53$ and $x = 1/2$ and we compare the performance of the `standard' FCCS rule (i.e., the rule analysed above) with an adaptive version where the approximation $\cQ^{r,d} \widehat{f}$ in
\eqref{FCCSmolyak_intro} is replaced by a  dimension-adaptive procedure.

\blue{
Our dimension-adaptive algorithm is implemented using the Sparse Grids Matlab kit \cite{nobile2016adaptive,piazzola2022sparse} and uses an adaptive procedure motivated by discussions in \cite{GeGr:03, nobile2016adaptive}. The complete adaptive procedure is given in Algorithm \ref{algo: adaptive FCCS rule}. In Algorithm \ref{algo: adaptive FCCS rule}, $\mathbf{G}, \mathbf{L}$ and $\mathbf{R_{L}}$ are all multi-index sets. A multi-index set $\mathbf{L}$ is said to be downward closed (see, e.g., \cite{DaPr:12})
if $\forall \bell \in \mathbf{L}, \bell - \mathbf{e}_j \in \mathbf{L}$ for $1 \le j \le d$ such that $\ell_j > 1$. Moreover, $\cI_{\mathbf{L}} f$ defined in Line 2 of Algorithm \ref{algo: adaptive FCCS rule} is the FCCS rule based on the generalized sparse grids over the multi-index set $\mathbf{L}$; see also \cite[Section 3.1]{GeGr:03}, \cite[Section 5.2]{GeGr:98} and \cite[Equations (3) and (4)]{nobile2016adaptive}. Then $\mathbf{G}$ contains the indices, over which the newest integral is computed, $\mathbf{L}$ contains the indices, which have been explored and whose neighbours (defined in 5 of of Algorithm \ref{algo: adaptive FCCS rule}) are to be explored, and $\mathbf{R_L}$ contains the indices, which have been explored and are candidates to be chosen to be added into $\mathbf{L}$. Also note that $\mathbf{G} = \mathbf{L} \cup \mathbf{R_L}$. For more detailed discussions of the dimension-adaptive algorithm, we refer readers to \cite{GeGr:03, nobile2016adaptive}.
}

\blue{
\begin{algorithm}[H]
\caption{\textbf{Adaptive FCCS rule}} \label{algo: adaptive FCCS rule}
\begin{algorithmic}[1]
\Statex \textbf{Inputs}: maximum number of sparse grids $N_{\max}$, tolerance $\tau$.
\State Let $\mathbf{L} = \{ \mathbf{1} \}, \mathbf{G} = \{ \mathbf{1} \}, \mathbf{R_{L}} =  \emptyset, \bell = \mathbf{1}$.
\State Let $\cI_{\text{old}} = \cI_{\mathbf{L}} f := \sum_{\mathbf{j} \in \mathbf{L}} \int_{[-1, 1]^d} (D^{j_1} \otimes \cdots \otimes D^{j_d}) \hf(\by) \exp(\ri k \ta \cdot \by) \rd \by = \sum_{\mathbf{j} \in \mathbf{L}} c_{\mathbf{j}} (I^{\omega_1, j_1} \otimes \cdots \otimes I^{\omega_d, j_d}) f$, where $c_{\mathbf{j}} = \sum_{\mathbf{k} \in \{0, 1\}^d, (\mathbf{k} + \mathbf{j}) \in \mathbf{L}} (-1)^{|\mathbf{k}|}$.
\State Let $\mathscr{P}_{\text{old}} = \text{pts}(\cI_{\text{old}})$, where $\text{pts}(\cI_{\text{old}})$ is the set of points used in computing $\cI_{\text{old}}$, and $N = | \mathscr{P}_{\text{old}} |$, where $| \mathscr{P}_{\text{old}} |$ is the cardinality of $\mathscr{P}_{\text{old}}$, and $P_{\max} = +\infty$.
\While {$N < N_{\max}$ \textbf{and} $P_{\max} \ge \tau$}
\State $\boldsymbol{\mathcal{N}_g} = \text{neigh}(\bell, \mathbf{L}) := \{ \mathbf{j} \in \mathbb{N}^d \backslash \mathbf{L}: | \mathbf{j} - \bell | = 1 \}$.
\For {$\mathbf{j} \in \boldsymbol{\mathcal{N}_g}$ such that $\mathbf{L} \cup \{ \mathbf{j} \}$ is downward closed}
\State $\mathbf{G} = \mathbf{G} \cup \{ \mathbf{j} \}$.
\State $\cI_{\text{new}} = \cI_{\mathbf{G}} f$.
\State $\mathscr{P}_{\text{new}} = \text{pts}(\cI_{\text{new}})$.
\State $N = N + | \mathscr{P}_{\text{new}} \backslash \mathscr{P}_{\text{old}} |$.
\State Compute the \emph{local profit} for $\mathbf{j}$: $P_{\mathbf{j}} = | \cI_{\text{new}} - \cI_{\text{old}} | / | \cI_{\text{new}} |$.
\State $R_{\mathbf{L}} = R_{\mathbf{L}} \cup \{ \mathbf{j} \}$.
\State $\cI_{\text{old}} = \cI_{\text{new}}, \mathscr{P}_{\text{old}} = \mathscr{P}_{\text{new}}$.
\EndFor
\State Compute the \emph{global profit} $P_{\max} = \max \{ P_{\mathbf{j}}: \mathbf{j} \in \boldsymbol{\mathcal{N}_g} \text{ such that } \mathbf{L} \cup \{ \mathbf{j} \} \text{ is downward closed}\}$.
\State Choose the $\bell$ from $R_{\mathbf{L}}$ with the highest profit $P_{\bell}$.
\State $\mathbf{L} = \mathbf{L} \cup \{ \bell \}, R_{\mathbf{L}} = R_{\mathbf{L}} \backslash \{ \bell \}$.
\EndWhile
\State $\cI^{\tau} = \cI_{\text{new}}$.
\Statex \textbf{Outputs}: the integral $\cI^{\tau}$.
\end{algorithmic}
\end{algorithm}
}

To compute \blue{the error indicators}, a reference value for $\cI(x)$  is computed by `brute force',  using
the tensor product Gauss-Legendre rule with 25 Gauss points in each of the $d$ dimensions.

Results are given  in Tables \ref{tab: comparison of adaptive and standard FCCS rule when d is 4} -- \ref{tab: comparison of adaptive and standard FCCS rule when d is 8}. The tables show the substantial advantage of the adaptive method in terms of the number of function evaluations over the standard method when the dimensions have decreasing importance, a situation often encountered in UQ applications.
\begin{table}[H]
\begin{center}
\begin{tabular}{|c|c|c|c|c|}
\hline
 & adaptive & \multicolumn{3}{c|}{standard} \\
\hline
 & $\tau = 10^{-4}$ & $r = 4$ & $r = 5$ & $r = 6$ \\
\hline
\blue{$E^{k,d, \ba,r}(f)$} & 1.15 (-7) & 8.37 (-6) & 1.34 (-7) & 7.21 (-10) \\
\hline
number of function evaluations & 53 & 137 & 401 & 1105 \\
\hline
\end{tabular}
\caption{\label{tab: comparison of adaptive and standard FCCS rule when d is 4} Comparison of the dimension-adaptive and standard FCCS rule for  $\cI(1/2)$,  $d = 4$}
\end{center}
\end{table}

\begin{table}[H]
\begin{center}
\begin{tabular}{|c|c|c|c|c|}
\hline
 & adaptive & \multicolumn{3}{c|}{standard} \\
\hline
 & $\tau = 10^{-6}$ & $r = 4$ & $r = 5$ & $r = 6$ \\
\hline
\blue{$E^{k,d, \ba,r}(f)$} & 9.33 (-8) & 8.46 (-6) & 1.41 (-7) & 8.64 (-10) \\
\hline
number of function evaluations & 129 & 389 & 1457 & 4865 \\
\hline
\end{tabular}
\caption{\label{tab: comparison of adaptive and standard FCCS rule when d is 6} Comparison of the dimension-adaptive and standard FCCS rule  for $\cI(1/2)$,  $d = 6$}
\end{center}
\end{table}

\begin{table}[H]
\begin{center}
\begin{tabular}{|c|c|c|c|c|}
\hline
 & adaptive & \multicolumn{3}{c|}{standard} \\
\hline
 & $\tau = 10^{-6}$ & $r = 4$ & $r = 5$ & $r = 6$ \\
\hline
\blue{$E^{k,d, \ba,r}(f)$} & 1.17 (-7) & 8.46 (-6) & 1.41 (-7) & 7.85 (-10) \\
\hline
number of function evaluations & 151 & 849 & 3937 & 15713 \\
\hline
\end{tabular}
\caption{\label{tab: comparison of adaptive and standard FCCS rule when d is 8} Comparison of the dimension-adaptive and standard FCCS rules for $\cI(1/2)$,  $d = 8$}
\end{center}
\end{table}

\subsection{ UQ problem for the Helmholtz equation} \label{subsec: numerical examples of UQ problem}

In this subsection we consider the Helmholtz problem \eqref{Helmholtz_Eq_2} -- \eqref{Helmholtz_BC_2} with random  $n = n(x, \by)$ given by \eqref{eq: parametrization of refractive index}, but with $F$ a function of $x$ only,  so that the solution $u= u(x, \by)$ depends on $x$ and  $\by$ (and also on the  frequency $k$). Then formulae \eqref{intmu} -- \eqref{intnu} show that  $\mathbb{E}[u(x)]$ can be written as a sum of oscillatory integrals  with kernels  
  given in equation \eqref{app2}. These integrals  are (formally) in a form  suitable for approximation by our FCCS rule, but in order to predict more precisely how well this will work,  some further analysis is needed to investigate the
  regularity of $\widetilde{\mu}$ and $\widetilde{\nu}$ with respect to $\by$. Since the principal components of $\widetilde{\mu},\, \widetilde{\nu}$ (i.e. those components which are $\mathcal{O}(1)$ as $k \rightarrow \infty$) are $\mu_0, \nu_0$ respectively, we restrict the discussion  here to the analysis of the regularity of $\mu_0, \nu_0$ with respect to $\by$. In the following discussion we make the  simplifying  assumptions that
  (in the random problem),
  the Dirichlet data $u_L = u(0) $ is a constant independent of $\by$ and that
  $n_\infty =  n(1, \by)$ is a positive constant  for all $\by \in [-1,1]^d$. Then it can be shown,  after some algebra,  that
  \begin{align} \label{mu0nu0}
  \mu_0(x,\by) =  \alpha_0^1(\by) n(x, \by)^{-1/2} , \quad \text{and} \quad \nu_0(x,\by) =  \alpha_0^2(\by) n(x, \by)^{-1/2},
  \end{align}
  where the functions $\alpha_0^j$ are given by 
  \begin{align} \label{alphas1}
  \alpha_0^2(\by)  = \frac{u_L \sqrt{n(0,\by)} n'(1,\by) }{2 \ri} \frac{\exp(\ri k N(1, \by))}{ n'(1,\by) \sin(k N(1,\by))  - 2 k n_{\infty}^2\exp(-\ri k N(1,\by))}
  \end{align}
  and
  \begin{align} \label{alphas2} \alpha_0^1(\by) = u_L \sqrt{n(0,\by)} - \alpha_0^2(\by).
  \end{align}

  From this we  see that:
  \begin{itemize}
  \item[(i)] If $n'(1,\by) = 0$ (i.e., $n(x, \by)$ is a constant function of  $x$ near $x = 1$ for all $\by$),   then $\alpha_0^2(\by) = 0$,
      $\alpha_0^1(\by)  = u_L \sqrt{n(0,\by)}$, and there is no $k-$dependent oscillation with respect to $\by$  in $\alpha_0^1(\by)$. The component   $r_0$ in the expansion \eqref{utm} corresponds to a wave moving from left to right across the domain;
  \item[(ii)] If $n'(1,\by) \not= 0$ then, while $\alpha_0^1$ and $\alpha_0^2$ are both potentially have $k-$dependent  oscillations with respect to  $\by$, the {\em amplitude} of their oscillatory components decays with $\mathcal{O}(1/k)$  as $k$ increases.
    \end{itemize}
    These facts  allow us to apply the FCCS rule directly to the integrals  \eqref{intmu} and \eqref{intnu} without any further splitting of their kernels.
    We do this in the following example for a case where $n'(1,\by)$ is not the zero function, and observe  good results.

    \bigskip


In the following two examples we consider computing $\mathbb{E}[u(1)]$ for the problem \eqref{Helmholtz_Eq_2} -- \eqref{Helmholtz_BC_2} with
$u_L = 1, n_{\infty} = 1, F(x) = x$, with  $n$ given by \eqref{eq: parametrization of refractive index},
and \eqref{def_n}. So in this case $n(1,\by) = n_\infty = 1$ for all $\by$ and formulae \eqref{alphas1}, \eqref{alphas2} hold.

Using the asymptotic approximation described in \S \ref{subsec: asymptotic expansion}, this can be     approximated by
 $\mathbb{E}[\widetilde{u}^1(1)]$, where $\widetilde{u}^1$ is defined in Remark \ref{rem:app}.
 The formulae in   \eqref{intmu}, \eqref{intnu} and \eqref{intn} show that $\mathbb{E}[\widetilde{u}^1(1)]$ can be written as a sum of three multidimensional integrals, the first two of which are oscillatory.

 The functions  $\widetilde{\mu}$ and $\widetilde{\nu}$ appearing in \eqref{intmu} and \eqref{intnu} are
 obtained using the formulae \eqref{app1},  \eqref{app2}, \eqref{tot_soln}, requiring the solution of a  system of ODEs that are non-oscillatory with respect to $x$.
   To solve these we use the method described in
 \S \ref{subsec: computation of mu and nu} to do this with parameters chosen as
 $M = 1, L = 1024$ and $M_G = 10$. These parameters are chosen to give very accurate values of $\widetilde{\mu}$ and $\widetilde{\nu}$ and it is not the purpose of this paper to investigate the most efficient choice of these parameters, since this question is not related to our main task here, namely to find methods which are efficient in terms of $k$ and $d$ dependence.

 We then study the performance of both the standard and the adaptive methods for approximating the integrals appearing in  \eqref{intmu}, \eqref{intnu} and \eqref{intn}.  

 \paragraph{Example 7 - The standard FCCS  method.}

 For $r\geq 1$,   we denote by {$\mathbb{E}^{k,r}[\wu^1(x)]$ the approximation of $\mathbb{E}[\wu^1(x)]$ obtained by applying the FCCS rule to  \eqref{intmu}, \eqref{intnu} and \eqref{intn} with maximum level $r$}. Since \eqref{intn} is not oscillatory, the FCCS rule just corresponds to a standard sparse grid quadrature on the hierarchy of grids \eqref{eq:grids}.


\begin{table}[H]
\begin{center}
\begin{tabular}{l|llll}
$r$ & $k = 8$ & $k = 16$ & $k = 32$ & $k = 64$ \\
\hline
  5&	5.86 (-3)&	1.18 (-4)&	{8.83 (-4)}&	{8.09 (-4)}\\
   6&	5.83 (-3)&	{2.48 (-5)}& {5.06 (-5)}&	{3.19 (-4)}\\
  7 &	5.83 (-3)&	{2.78 (-5)}& {6.89 (-6)}&	{1.13 (-4)}\\
 8 &	5.83 (-3)&	{2.79 (-5)}&	{5.83 (-6)}&	{1.85 (-6)}\\
  9&	5.83 (-3)&	{2.79 (-5)}&	{5.79 (-6)}&	{8.52 (-7)}\\
10&	5.83 (-3)&	{2.79 (-5)}&	{5.79 (-6)}&	{3.65 (-7)}\\
  11&	5.83 (-3)&	{2.79 (-5)}&	{5.79 (-6)}&	{3.50 (-7)}\\
  12& 	5.83 (-3)&	{2.79 (-5)}&	{5.79 (-6)}&	{3.50 (-7)}\\
\blue{$\alpha$} & & {7.71} & {2.27} & {4.05} \\
\hline
\end{tabular}
\end{center}
\caption{\label{tab:absolute error of UQ problem where d is 4} $\vert \mathbb{E}^{k,r}[\wu^1(1)] - \mathbb{E}[u(1)]\vert$  for $d = 4$ as  $r$ and $k$ vary. \blue{The last row shows the decay rate $\mathcal{O}(k^{-\alpha})$ with $\alpha$ computed from successive pairs of errors of $r = 12$.}}
\end{table}

 %





  We first consider $d = 4$. A reference value
  is computed by applying the continuous piecewise linear  finite element method to the full $k-$dependent boundary-value problem \eqref{Helmholtz_Eq_2} -- \eqref{Helmholtz_BC_2} with  spatial mesh size $h = (2^{16} + 1)^{-1}$. This is done for sample points $\by$ chosen on the   grid  formed as the tensor product of the 1d Gauss-Legendre rule with $50$ Gauss points in each of the $d$  dimensions. This is an expensive method, but it provides a very accurate  $\mathbb{E}[u(1)]$, and is done only once to allow us to study errors.   The absolute error
  $\vert \mathbb{E}^{k,r}[\wu^1(1)] - \mathbb{E}[u(1)]\vert$ is shown in Table \ref{tab:absolute error of UQ problem where d is 4}.
  To verify the reliability of these errors we have repeated the experiment again with $h = (2^{17} + 1)^{-1}$ and found that the absolute differences between the reference solutions computed by $h = (2^{16} + 1)^{-1}$ and $h = (2^{17} + 1)^{-1}$ are smaller than the errors shown in Table \ref{tab:absolute error of UQ problem where d is 4}.

  Recall that the method we are studying has an error with respect to $k$  (due to the asymptotic approximation) and with respect to $r$ (from the sparse grid approximation). Hence we see convergence as both $k$ and $r$ increase by reading diagonally across the table, e.g., starting from   $r=7$ and $k=8$ we see the sequence: $5.83(-3),  \ {2.79(-5)}, \ {5.79(-6)}, \ {3.65(-7)}$, and similarly for other diagonals.

Since, for  $r = 12$ there is not much  error in the oscillatory integrals, we see  decay of the error as $k$ increases. For small $r$,  on the other hand, the rows of Table  \ref{tab:absolute error of UQ problem where d is 4} do not exhibit steady decay with respect to $k$ due to the error in the oscillatory integrals.

The computation of the exact reference value used in Table \ref{tab:absolute error of UQ problem where d is 4} is costly and not feasible for higher dimensions or wavenumbers. 
Instead, in  Table \ref{tab:7}  we study the error proxy
\begin{align} \label{prox2}
\vert \mathbb{E}^{k,r}[\widetilde{u}^1(1)] - \mathbb{E}^{k,r+4}[\widetilde{u}^1(1)] \vert
\end{align}
for $d=4$ and  higher values of $k$. Table \ref{tab:absolute error of UQ problem where d is 4} tells us that we should use this proxy cautiously,  since (for example)  computing this quantity for the column corresponding to $k=8$ will give values uniformly of order $10^{-6}$ where the true error is much larger.
However   reading Table \ref{tab:7} diagonally we still see quite  convincing  convergence of this proxy to zero as $k,r$ both increase simultaneously, although
the convergence is not always monotonic.



\begin{table}[H]
\begin{center}
\begin{tabular}{l|lll}
$r$   & $k = 32$ & $k = 64$ & $k = 128$  \\
\hline
4  & 2.17 (-3) & 5.35 (-4) & 4.04 (-5) \\
5 &   8.77 (-4) & 8.09 (-4) & 5.43 (-5)  \\
6 &   4.48 (-5) & 3.19 (-4) & 1.02 (-4)  \\
7 &   2.22 (-6) & 1.13 (-4) & 5.19 (-5)  \\
8 &  1.30 (-7) & 1.50 (-6) & 5.84 (-5)  \\
\hline
\end{tabular}
\end{center}
\caption{\label{tab:7} Values of the `Error proxy' \eqref{prox2} for $d = 4$ for various $r$ and $k$}
\end{table}

In Table \ref{tab:6} 
we study the `error proxy':
\begin{align} \label{prox1}
\vert \mathbb{E}^{k,r}[\widetilde{u}^1(1)] - \mathbb{E}^{k,r+2}[\widetilde{u}^1(1)] \vert
\end{align}
for the case $d = 6$, and observe a similar diagonal behaviour.


\begin{table}[H]
\begin{center}
\begin{tabular}{l|lll}
$r$ &  $k = 32$ & $k = 64$ & $k = 128$  \\
\hline
4 &  2.21 (-3) & 2.05 (-4) & 1.25 (-4)  \\
5 &  8.89 (-4) & 9.17 (-4) & 4.03 (-5)  \\
6 & 4.20 (-5) & 3.31 (-4) & 1.54 (-4)  \\
7 & 2.12 (-6) & 1.08 (-4) & 5.13 (-5)  \\
8 &  1.20 (-7) & 1.71 (-6) & 5.63 (-5)  \\
\hline
\end{tabular}
\end{center}
\caption{\label{tab:6} Values of the `Error proxy' \eqref{prox1} for $d = 6$ for various $r$ and $k$}
\end{table}

%

\paragraph{Example 8 - The dimension adaptive algorithm.}

Finally we consider the dimension adaptive method for the UQ problem. In this case, for any given $k$, $\mathbb{E}[u(1)]$ is computed by approximating the three integrals \eqref{intmu}, \eqref{intnu} and \eqref{intn},  using the dimension adaptive method. To do this we introduce a  tolerance parameter $\tau$. {Since the formulae \eqref{alphas1}, \eqref{alphas2} show that the amplitude of $\nu_0$ appearing in the integral \eqref{intnu}
   decays with $\mathcal{O}(1/k)$, and (recall Example 6 above), the adaptive method aims to control the \blue{normalized error} (and not the absolute error) in the approximate integral, we use a smaller tolerance $\tau$ for integrals  \eqref{intmu} and \eqref{intn} and larger tolerance $k \tau$ for integral \eqref{intnu}, {and the resulting approximation of $\mathbb{E}[\wu^1(x)]$ is denoted by $\mathbb{E}^{k, \tau}[\wu^1(x)]$}.} 
For $d = 6, 8, 10$, we display values of the error proxy:
\begin{align} \label{prox3}
\vert \mathbb{E}^{k,\tau}[\wu^1(1)] - \mathbb{E}^{k, \tau/4}[\wu^1(1)] \vert.
\end{align}
We also let $N_{\widetilde{\mu}}, N_{\widetilde{\nu}}, N_{\widetilde{f}}$ denote, respectively,
the number of grid points in the adaptive sparse grids used for
computing integrals \eqref{intmu}, \eqref{intnu} and \eqref{intn} and set $N_{tot} = N_{\widetilde{\mu}} + N_{\widetilde{\nu}} +N_{\widetilde{f}}$.



\begin{table}[H]
\begin{center}
\begin{tabular}{l|llll}
$\tau$ & $k = 32$ & $k = 64$ & $k = 128$ & $k = 256$ \\
\hline
0.01 & 3.84 (-5) & 1.13 (-4) & 2.77 (-5) & 2.78 (-7) \\
0.005 & 6.37 (-5) & 1.86 (-4) & 2.57 (-5) & 7.76 (-6) \\
0.0025 & 7.01 (-5) & 5.98 (-5) & 3.12 (-5) & 3.19 (-7) \\
0.00125 & 2.64 (-4) & 1.81 (-6) & 6.63 (-5) & 7.19 (-6) \\
\hline
\hline
0.01 & (21, 27, 13, 61) & (21, 75, 13, 109) & (13, 15, 13, 41) & (13, 13, 13, 39) \\
0.005 & (49, 43, 13, 105) & (21, 75, 13, 109) & (13, 21, 13, 47) & (21, 13, 13, 47) \\
0.0025 & (53, 53, 13, 119) & (81, 75, 13, 169) & (13, 149, 13, 175) & (21, 15, 13, 49) \\
0.00125 & (53, 77, 13, 143) & (141, 81, 13, 235) & (21, 149, 13, 183) & (39, 15, 13, 67) \\
0.000625 & (53, 77, 13, 143) & (149, 141, 13, 303) & (31, 157, 13, 201) & (55, 15, 13, 83) \\
0.0003125 & (91, 77, 13, 181) & (219, 141, 13, 373) & (167, 277, 13, 457) & (55, 45, 13, 113) \\
\hline
\end{tabular}
\end{center}
\caption{\label{tab: modified adaptive grids of UQ problem where d is 6} Results of the dimension adaptive method for $d = 6$ for various $\tau$ and $k$. Top panel: Values of the `Error proxy' \eqref{prox3}. Bottom panel: $(N_{\widetilde{\mu}}, N_{\widetilde{\nu}}, N_{\widetilde{f}}, N_{tot})$}
\end{table}

\begin{table}[H]
\begin{center}
\begin{tabular}{l|llll}
$\tau$ & $k = 32$ & $k = 64$ & $k = 128$ & $k = 256$ \\
\hline
0.01 & 3.84 (-5) & 1.13 (-4) & 2.77 (-5) & 2.78 (-7) \\
0.005 & 6.37 (-5) & 1.86 (-4) & 2.57 (-5) & 7.76 (-6) \\
0.0025 & 7.01 (-5) & 5.98 (-5) & 3.12 (-5) & 3.19 (-7) \\
0.00125 & 2.64 (-4) & 1.81 (-6) & 6.63 (-5) & 7.19 (-6) \\
\hline
\hline
0.01 & (25, 31, 17, 73) & (25, 79, 17, 121) & (17, 19, 17, 53) & (17, 17, 17, 51) \\
0.005 & (53, 47, 17, 117) & (25, 79, 17, 121) & (17, 25, 17, 59) & (25, 17, 17, 59) \\
0.0025 & (57, 57, 17, 131) & (85, 79, 17, 181) & (17, 153, 17, 187) & (25, 19, 17, 61) \\
0.00125 & (57, 81, 17, 155) & (145, 85, 17, 247) & (25, 153, 17, 195) & (43, 19, 17, 79) \\
0.000625 & (57, 81, 17, 155) & (153, 145, 17, 315) & (35, 161, 17, 213) & (59, 19, 17, 95) \\
0.0003125 & (95, 81, 17, 193) & (223, 145, 17, 385) & (171, 281, 17, 469) & (59, 49, 17, 125) \\
\hline
\end{tabular}
\end{center}
\caption{\label{tab: modified adaptive grids of UQ problem where d is 8} Results of the dimension adaptive method for $d = 8$ for various $\tau$ and $k$. Top panel: Values of the `Error proxy' \eqref{prox3}.  Bottom panel: $(N_{\widetilde{\mu}}, N_{\widetilde{\nu}}, N_{\widetilde{f}}, N_{tot})$ }
\end{table}

\begin{table}[H]
\begin{center}
\begin{tabular}{l|llll}
$\tau$ & $k = 32$ & $k = 64$ & $k = 128$ & $k = 256$ \\
\hline
0.01 & 3.84 (-5) & 1.13 (-4) & 2.77 (-5) & 2.78 (-7) \\
0.005 & 6.37 (-5) & 1.86 (-4) & 2.57 (-5) & 7.76 (-6) \\
0.0025 & 7.01 (-5) & 5.98 (-5) & 3.12 (-5) & 3.19 (-7) \\
0.00125 & 2.64 (-4) & 1.81 (-6) & 6.63 (-5) & 7.19 (-6) \\
\hline
\hline
0.01 & (29, 35, 21, 85) & (29, 83, 21, 133) & (21, 23, 21, 65) & (21, 21, 21, 63) \\
0.005 & (57, 51, 21, 129) & (29, 83, 21, 133) & (21, 29, 21, 71) & (29, 21, 21, 71) \\
0.0025 & (61, 61, 21, 143) & (89, 83, 21, 193) & (21, 157, 21, 199) & (29, 23, 21, 73) \\
0.00125 & (61, 85, 21, 167) & (149, 89, 21, 259) & (29, 157, 21, 207) & (47, 23, 21, 91) \\
0.000625 & (61, 85, 21, 167) & (157, 149, 21, 327) & (39, 165, 21, 225) & (63, 23, 21, 107) \\
0.0003125 & (99, 85, 21, 205) & (227, 149, 21, 397) & (175, 285, 21, 481) & (63, 53, 21, 137) \\
\hline
\end{tabular}
\end{center}
\caption{\label{tab: modified adaptive grids of UQ problem where d is 10} Results of the dimension adaptive method for $d = 10$ for various $\tau$ and $k$. Top panel: Values of the `Error proxy' \eqref{prox3}. Bottom panel: $(N_{\widetilde{\mu}}, N_{\widetilde{\nu}}, N_{\widetilde{f}}, N_{tot})$ }
\end{table}

In Tables \ref{tab: modified adaptive grids of UQ problem where d is 6} -- \ref{tab: modified adaptive grids of UQ problem where d is 10} we again observe (mostly) diagonal convergence as $\tau \rightarrow 0$ and $k \rightarrow \infty$. For fixed $k$ and $\tau$ we see only very modest growth in the amount of work as the dimension increases. In fact for $\tau = 0.00125$ and any fixed $k$ a linear least squares fit on the data here  suggests the number of function evaluations grows at most like $\mathcal{O}(d^{0.1})$ as $d$ increases.
For fixed $\tau$ and $d$, we see some initial growth of the work as $k$ increases, but this seems to reduce substantially as $k$ gets higher.

\section{Conclusion}
\label{sec:conclusion}
\blue{We have proposed and analysed a Filon-Clenshaw-Curtis-Smolyak sparse grid-based rule for $d-$dimensional oscillatory integrals $ \int_{[{-1},1]^d} f(\by) \exp(\ri k \ba . \by) \rd \by $,
which is accurate for all $k >0$, where  $\ba \in \mathbb{R}^d$ may have components with large,  small or  vanishing modulus.
We give novel upper bounds for the absolute error. These take the form of a Smolyak-type
error estimate multiplied by the factor  $k^{-\widetilde{d}}$, where $\widetilde{d}$ denotes the number of oscillatory dimensions. The Smolyak estimate is explicit in $N$ (the number of function evaluations), the maximum level $r$ of the  sparse grids and the regularity of $f$ in spaces of functions with continuous mixed derivatives,
The estimates imply that the normalized error (defined here to be the absolute error multiplied by $k^{\widetilde{d}}$) converges quickly  as $N$ grows, (assuming  enough regularity of $f$) and is bounded with respect to $k$. On the other hand if all components of $\ba$ are non-vanishing,  we also show that the normalized error decays with order $\mathcal{O}(k^{-1})$,  as $k \rightarrow \infty$.}

\blue{We apply the rule to the forward problem of UQ for a one-space-dimensional Helmholtz problem with a random refractive index $n$, depending in an affine way on
$d$ i.i.d. uniform random variables. An asymptotic approximation  shows,
that for large $k$,  expectations of linear functionals of the solution can be formulated as a sum of multidimensional oscillatory integrals which are computed using  the FCCS rule. Numerical experiments show that this yields a UQ algorithm with increasing accuracy as $k$ increases.}

\blue{There are a number of interesting directions for possible future work. It seems feasible to extend to the case of additively separable  non-linear phase, where $\ba \cdot \by$ is replaced by $g(\by) = g_1(y_1) + \ldots + g_d(y_d)$, and each $g_i$ may have stationary points or singularities.  Then one could consider  incorporating  methods for 1D oscillatory integrals with nonlinear phase (such as  \cite{GrahamDominguez:2013, olver2007moment, DeHuIs:18, MaIsPe:22}) in the dimensions with the stationarity. More general nonlinear phases with stationary points would present a bigger challenge. It would be interesting, also, to investigate whether the asymptotic method for the 1D Helmholtz problem could be replaced by a practical  `hybrid' method, based on a solution ansatz like  \eqref{app1}, but  for which solutions are accurate for all ranges of $k$. The extension of the asymptotic method to higher physical space dimension represents a particular challenge, since the current method makes use of the fact that in 1D waves can only move from left-to-right or right-to-left along the domain.  On the other hand the 1D in space model is still interesting in the UQ context; in fact  this  is the first time that the  oscillatory behaviour of Helmholtz solutions with respect to random perturbations of the refractive index has been explicitly displayed and exploited. }

\section*{Acknowledgement}
\noindent
The research of I.G. Graham is supported by UK EPSRC grant EP/S003975/1. The research of Z. Zhang is supported by Hong Kong RGC grants (Projects 17300318 and 17307921), National Natural Science Foundation of China (Project 12171406), and Seed Funding for Strategic Interdisciplinary Research Scheme 2021/22 (HKU). The computations were performed using research computing facilities offered by Information Technology Services, the University of Hong Kong.

We wish to thank Prof V\'{i}ctor Dom\'{i}nguez (Public University of Navarra) for sharing with us his codes for implementing the 1D  Filon-Clenshaw-Curtis rules.

\blue{We also thank the anonymous referees for their insightful  comments on the original version of this paper.}

\bibliographystyle{siam}
\bibliography{ZWpaper_Zhizhang}



\end{document}